\documentclass[a4paper,12pt]{amsart}
\usepackage{amsmath,amssymb,colordvi}
\usepackage{multirow}
\usepackage{verbatim}
\newtheorem{theorem}{Theorem}
\newtheorem{corollary}[theorem]{Corollary}

\newtheorem{proposition}[theorem]{Proposition}
\newtheorem{definition}[theorem]{Definition}

\newtheorem{remark}[theorem]{\it Remark}

\DeclareMathOperator*{\esssup}{ess\,sup}

\def\B{\mathbb B}
\def\N{\mathbb N}

\def\NN{\mathbb N}
\def\CC{\mathbb C}
\def\CC{\mathbb C}
\def\DD{\mathbb D}
\def\R{\mathbb R}
\def\RR{\mathbb R}
\def\ZZ{\mathbb Z}

\def\P{\mathcal P}
\def\Zeta{\mathcal Z}
\def\L{\mathcal L}
\def\Ca{\mathcal C_\alpha}
\begin{document}

\title[Two Koopman semigroups on discrete Lebesgue spaces]
{Two Koopman semigroups on discrete Lebesgue spaces }

\author{Pedro J. Miana}
\address{Departamento de
Matem\'aticas, Instituto Universitario de Matem\'aticas y Aplicaciones, Universidad de Zaragoza,  50009 Zaragoza, Spain.}
 \email{pjmiana@unizar.es}

\thanks{Author has been partially supported by PID2022-137294NB-I00, DGI-FEDER, of the MCEI and Project E48-20R, Gobierno de Arag\'on, Spain}


\dedicatory{To my dear father Francisco Miana Garc{\'i}a}



\begin{abstract} In this paper we are interested to connect Koopman  semigroups in Lebesgue funcion spaces $L^p(\R^+)$ and $C_0$-semigroups in Lebesgue sequence spaces $\ell^p$ for $1\le p < \infty$. To get this we use certain Poisson
transformation
 ${\P}: L^p(\mathbb{R}^+)\to \ell^p$ and its adjoint ${\P}^*$ which allows carry semigroup properties from one space to the other one. Two Koopman  semigroups on $\ell^p$ are presented and linked to the standard Koopman semigroup $T_p(t)f(r):= e^{-{t\over p}}f(e^{-t}r)$ and
$S_{p}(t)f(r):= e^{-t\over p}f(e^{-t}r+1-e^{-t})$ for $t,r>0$ on $L^p(\R^+)$. In the last section we  introduce Ces\`aro-like operators subordinated to these Koopman semigroups on  $\ell^p$.
\end{abstract}

\date{}

\maketitle

\section*{Introduction}

\setcounter{theorem}{0}
\setcounter{equation}{0}

Composition operators allow to face up fundamental questions about certain operators with elegant classical results from functional analysis, complex variable, and harmonic analysis.  In dynamic systems, composition operators are usually named as Koopman operators, in honor of the Franco-American mathematician Bernard O. Koopman. Koopman introduced these operators  to join classical Hamiltonian mechanics with the theory of Hilbert spaces and their linear transformations (\cite{Koopman}). Within this framework, the canonical resolution of the identity (usually called the ``spectrum of the dynamical system") is introduced and described in terms of a one-parameter group of unitary (and composition) operators in Hilbert space.

Take $K$ a compact and $X$ a locally compact spaces.  A mapping $\phi: [0,\infty)\,\,  \times K \to K$ is said a semiflow if for each $t\ge 0$ the mapping $\phi_t$ defined by $\phi_t(x):=\phi(t,x)$ is continuous, $\phi_0(x)=x$  and $\phi_s\circ \phi_t= \phi_{s+t}$ for $s,t\ge 0$.

In this topic, there is an abundance of literature available. The classic monographic \cite{Nagel} develops a deep  treatment of  positive semigroups in lattice spaces, and it characterizes continuous flows through semigroups of composition operators and their infinitesimal generators. A precise study of Koopman operators in $L^p({\bf X})$, with applications to ergodicity in measure-preserving systems, can be found in \cite{EFHN}.  In a finite measure space, reference \cite{EGK} characterizes  Koopman semigroups on $L^p({\bf X})$ for measure preserving semiflows. Furthermore, it shows that  measure-preserving flows on standard probability spaces and continuous flows on compact Borel probability spaces are equivalent. Recently  three specific semiflows  $\phi_t, \psi_t,\varphi_t:[0,\infty)\to [0,\infty)$ on the real half-line given by
$$
\phi_t(r):=e^{-t}r+1-e^{-t},\
\psi_t(r):= \frac{e^t r}{1+r(e^t-1)}, \  \varphi_t(r):= \frac{(1 + e^t)r -1 + e^t}{(-1 + e^t) r + 1 + e^t},
$$
for $t,r>0$ are considered in \cite{MP} to define and study weight Koopman semigroups  on $L^p(\RR^+)$ and some additional Sobolev-Lebesgue spaces.

%
%
%

On the other hand, the theory of continuous one-parameter semigroups of analytic self-maps on the unit disk has received significant attention in last decades. Semigroups of holomorphic self-maps on the unit disk (or on the complex half-plane or other domains) are closely connected with the theory of composition operators due to each one-parameter semigroup of holomorphic self-maps induces a semigroup of composition operators on certain holomorphic function spaces. This connection transfers functional analytic properties, such as compactness, periodicity or spectral properties, to corresponding dynamical questions about semigroups. The excellent monograph \cite{BCD} gives a thorough and comprehensive overview of the current state of the art in this interesting and complicated  field. Several well-known examples of semigroups of holomorphic semiflows are discussed in the excellent survey \cite{Si}.

Corresponding results on discrete Lebesgue space $\ell^p$ are unknown. Moreover, a rigorous definition of semiflows and  Koopman semigroups on $\ell^p$ seem impossible to achieve. However we present in this paper (sections 3.2 and 4) two concrete examples, $(\frak{T}_{p}(t))_{t>0}$ and $(\frak{T}_{\Delta,p}(t))_{t>0},$ of $C_0$-semigroups  on $\ell^p$ which are strongly linked to  Koopman semigroups $T_p(t)f(r):= e^{-{t\over p}}f(e^{-t}r)$ and
$S_{p}(t)f(r):= e^{-t\over p}f(e^{-t}r+1-e^{-t})$ for $t,r>0$ on $L^p(\R^+)$. It look likes natural consider these $C_0$-semigroups  as canonical examples of ``Koopman semigroups'' on $\ell^p$.

To establish this connection from $L^p(\RR^+)$ to $\ell^p$, we consider the Poisson transformation ${\P}: L^p(\mathbb{R}^+)\to \ell^p$ defined by
$$
{\P}(f)(n):=\int_0^\infty f(t) e^{-t}{t^n\over n!}dt, \qquad n \ge 0, \quad f\in L^p(\mathbb{R}^+),
$$
and its adjoint operator ${\P}^*:  \ell^p \to L^p(\mathbb{R}^+)$ given by
$$
{\P}^*(a)(t):=e^{-t}\sum_{n=0}^\infty a(n){t^n\over n!}, \qquad t\in \R^+, \quad a=(a(n))_{n\ge 0}\in \ell^p.
$$
Among other fields, both operators have appeared  in functional analysis, approximation theory and stochastic calculus. For example, the classical  Post-Widder operator (also called Gamma operator in \cite{LM} and in later references),
$$
P_{\lambda}(f,u):={(\lambda/u)^{\lambda}\over \Gamma(\lambda)}\int_0^\infty e^{-(\lambda t/u)} t^{\lambda-1} f(t)dt= {1\over \Gamma(\lambda)}\int_0^\infty e^{-s} {s^{\lambda-1}}f({u\over \lambda}s)ds, \qquad u, \lambda\in (0,\infty),
$$
 (\cite[p. 288]{Wi}) includes the Poisson transformation $\P$, i.e, $P_{n+1}(f,n+1)= {\P}(f)(n)$ for $n>0$. The Sz\'asz-Mirakyan operator
 $$
S_x(f,t):=e^{-tx}\sum_{n=0}^\infty f({n\over x}){(tx)^n\over n!}, \qquad t\ge 0, x>0,$$
 is a positive linear operator which uniformly approximates a certain class of continuous functions on the half line (\cite{St, Sz}). Note that
$S_1(f,t)= \P^*(\tilde f)(t)$ for $t\ge 0$ where the sequence $\tilde f$ is given by  $\tilde f(n):=f(n)$ for $n\ge 0$. In other context, the analysis of the dynamics of the semigroup solutions of certain partial differential equations are studied  on  Herzog type space of analytic functions defined by
$$
X_\rho:=\left\{ f: \RR\to \CC\, \, :\,\, f(x)=\sum_{n=0}^\infty {a_n\rho^n\over n!}x^n, \quad (a_n)_{n\ge 0}\in c_0(\NN_0)\right\},
$$
for some $\rho>0$, see for example \cite{H} and later references.

The article is self-contained and includes basic results that aid in its understanding. It is organized as follows. In the first section, we consider a modified Beta  function defined by
$$
\B_1(u,v):= \int_0^1(1-t)^{u-1}t^{v-1}e^{-t}dt, \qquad u, v\in \CC^+.
$$
and  a family of Poisson-Charlier orthogonal polynomials given by the recurrence relation $$p_{n+1}(z):= (z+n+1)p_n(z)-np_{n-1}(z), \qquad n\ge 2,$$ $p_0(z):=1 $ and
$p_1(z):=z+1$. These functions are connect with the resolvent and the point spectrum of certain finite difference operators, see Theorem \ref{resolvent} (i) and (iii).

In the second we  study the Poisson transformation ${\P}$ and its adjoint ${\P^*}$ which we have mentioned above. We have decided to follow the terminology presented in \cite{Li} with is based in the Poisson distribution
$$
{\frak p}_n(t):=e^{-t}{t^n\over n!}, \qquad t\ge 0, \quad n \in \NN_0.
$$
Algebraic properties of operators $\P$ and $\P^*$ are shown in Theorem \ref{convoss} and its connections with the Laplace and Zeta transform in Proposition \ref{conne}.

Two different families of $C_0$-semigroups in the Lebesgue sequence space $\ell^p$ are treated in the third section. First examples are convolution semigroups generated by forward difference operator $\Delta(a)(n):=a(n+1)-a(n)$ and the classical backward difference operator $\nabla(a)(n):=a(n-1)-a(n)$. These semigroups are well-known and can be founded in a large number of monographics and papers about finite differences, see for example \cite{GLM} and reference therein. Note that  $\nabla(\P(f))=-\P(f')$ and  $\P^*(\Delta (a))= (\P^*(a))'$, see Theorem \ref{eliz}.

First examples of Koopman semigroups  on $\ell^p$,  $(\frak{T}_{p}(t))_{t>0}$ and  $(\frak{S}_{p}(t))_{t>0},$ are generated by
\begin{eqnarray*}
 \frak{A}_p(a)(n)&:=&n (a(n-1)-a(n))-{1\over p}a(n),  \, \, n\ge 1;  \qquad \frak{A}_p(a)(0):=- {1\over p}a(0); \cr
\frak{B}_p(a)(n)&:=&(n+1)a(n+1)-na(n)-(1-{1\over p})a(n), \, \, n\ge 0,
\end{eqnarray*}
and were  introduced in \cite{AM2018}. Note that these operators are densely defined on $\ell^p$ and does not commute with $\Delta$ and $\nabla$ respectly (Remark \ref{notcommute}),  $\frak{B}_p(\P(f))= -\P(\Lambda_p(f))$ and $\P^*(\frak{A}_p(a))= \Lambda_p(\P^*(a))$ where $\Lambda_p$ is the infinitesimal generator of Koopman semigroup $T_p(t)f(r):= e^{-{t\over p}}f(e^{-t}r)$ on $L^p(\R^+)$.

A second example of Koopman semigroups, $(\frak{T}_{\Delta,p}(t))_{t>0}$ and $(\frak{S}_{\nabla,p}(t))_{t>0}$  on $\ell^p$ is presented in forth section. Its infinitesimal generator ${\frak A}_{\Delta, p}$ and ${\frak B}_{\nabla, p}$ are bounded perturbations of ${\frak A}_{p}$ and ${\frak B}_ p$, i.e., ${\frak A}_{\Delta, p}= {\frak A}_{p}+\Delta$ and ${\frak B}_{\Delta, p}= {\frak B}_{ p}+\nabla$ (Theorem \ref{inge}). The theory of general perturbed $C_0$-semigroup has received a special attention in the literature, see \cite[Chapter III]{En-Na-00} and reference therein.
Although  in general the explicit expression of the perturbed $C_0$-semigroup is not known,  we give them in this case in Definition \ref{deff}. Note that   $ \P^*\circ\frak{T}_{\Delta,p}(t)=  S_p(t)\circ \P^*$ for $t>0$ where $S_{p}(t)f(r):= e^{-t\over p}f(e^{-t}r+1-e^{-t})$ on $L^p(\R^+)$ (Theorem \ref{conmutante}).

In Table 1 we collect  $C_0$-semigroups and its infinitesimal generators which are considered  in sections \ref{savi}, \ref{savi2} and \ref{savi3}.

\begin{table}[h]\label{table1}
\begin{center}
 \caption{$C_0$-semigroups and its infinitesimal generators on $\ell^p$  }
\begin{tabular}{ |p{1.25cm}||p{11cm}|p{2cm}| }
 \hline
 $T(t)$& $T(t)a(n)$ & $A$   \\
 \hline
$e^{t\Delta }$   &  $e^{-t}\sum_{j\ge 0}^\infty a(j+n){t^j\over j!}$ &$\Delta$  \\
\hline
$e^{t\nabla}$   &  $e^{-t}\sum_{j\ge 0}^n a(j){t^{n-j}\over (n-j)!}$ &$\nabla$    \\
\hline
$\frak{T}_p(t)$   &  $e^{-{t\over p}}\sum_{j=0}^n{n \choose j}e^{-tj}(1-e^{-t})^{n-j}a(j)$   &$\frak{A}_p$ \\
\hline
$\frak{S}_p(t)$   &  $e^{-t(1-{1\over p})}e^{-tn}\sum_{j=n}^{\infty}{j\choose n}(1-e^{-t})^{j-n}a(j)$   &$\frak{B}_p$ \\
\hline
$\frak{T}_{\Delta,p}(t)$   &  $e^{-({t\over p}+1-e^{-t})}\sum_{j=0}^l{l \choose j}(1-e^{-t})^{l-j}e^{-tj}\sum_{n=j}^\infty{(1-e^{-t})^{n-j}\over (n-j)!}a(n)$   &${\frak A}_p+\Delta$ \\
\hline
$\frak{S}_{\nabla,p}(t)$   &  {$\small{e^{-(t(1-{1\over p})+1-e^{-t})}\sum_{j=0}^l{(1-e^{-t})^{l-j}\over (l-j)!}e^{-tj}\sum_{n=j}^\infty{n \choose j}(1-e^{-t})^{n-j}a(n)}$ }  &${\frak B}_p+\nabla$  \\
\hline
\end{tabular}
\end{center}
\end{table}

Similarly, we may collect also  $C_0$-semigroups and its infinitesimal generators on $L^p(\R^+)$ in Table 2. Connections between $C_0$-semigroups in $\ell^p$ and $L^p(\R^+)$ are presented in Theorems \ref{eliz}, \ref{intert2} and \ref{conmutante}.

\begin{table}[h]
\begin{center}
 \caption{$C_0$-semigroups and its infinitesimal generators on $L^p(\R^+)$   }
\begin{tabular}{ |p{1.25cm}||p{11cm}|p{2cm}| }
 \hline
 $T(t)$& $T(t)f(s)$ & $A$   \\
 \hline
$T_{left}(t)$   &  $f(s+t)$ &${d\over ds}$  \\
\hline
$T_{right}(t)$   &  $f(s-t)\chi_{(t,\infty)}(s)$ &${d^0\over ds}$    \\
\hline
$T_p(t)$   &  $e^{-{t\over p}}f(e^{-t}r)$   &$\Lambda_p$ \\
\hline
$S_{p}(t)$   &  $e^{-t\over p}f(e^{-t}s+1-e^{-t})$   &$\Lambda_p+{d\over ds}$ \\
\hline
$R_{p}(t)$   &  {$e^{t\over p}f(e^{t}s+1-e^{t})\chi_{(1-e^{-t}, \infty)}(s)$ }  &$-\Lambda_p+{d^0\over ds}$  \\
\hline
\end{tabular}
\end{center}
\end{table}

In the fifth section we consider the following integral subordination
$$
C^T_{\mu, \nu}(x):=\int_0^\infty e^{-\mu t} (1-e^{-t})^{\nu-1}T({t})(x)dt, \qquad x\in X, \qquad \mu, \nu>0,
$$
where $(T(t))_{t>0}$ is an uniformly bounded $C_0$-semigroup on a Banach space $X$. Then we prove directly that  ${\mathcal P}^*\circ\frak{C}_{\alpha}=  {\mathcal{C}}_{\alpha}\circ {\mathcal P}^*$ (Corollary \ref{kkl}) where
\begin{eqnarray*}
\Ca f (s)&=& \frac{\alpha}{s^\alpha} \int_0^s (s-u)^{\alpha-1} f(u) \, du, \quad s\ge 0,\cr
\frak{C}_{\alpha}a(n)&:=&\frac{\alpha}{\Gamma(\alpha+1+n)}\sum_{j=0}^n{\Gamma(n-j+\alpha)\over (n-j)!}a(j), \quad n\ge 0.
\end{eqnarray*}
We introduce operators $\frak{c}^{\Delta,p}_{\mu, \nu}$ and
$\frak{c}^{\nabla,p}_{\mu, \nu}$  and show that ${\mathcal P}\circ  C_{\mu, \nu}^{R_p}=\frak{c}^{\nabla,p}_{\mu, \nu}\circ {\mathcal P}$  and  ${\mathcal P}^*\circ\frak{c}^{\Delta,p}_{\mu, \nu}=  C_{\mu, \nu}^{S_p}\circ {\mathcal P}^*$ where $C_{\mu, \nu}^{S_p}$ and $C_{\mu, \nu}^{R_p}$
are Chen fractional integral, Corollary \ref{zzz}.

\medskip

\noindent{\bf Notation.} The set $\CC_{-}=\{z\in \CC\,\, :\,\, \Re z<0\}$ and $\CC_{+}=\{z\in \CC\,\, :\,\, \Re z>0\}$. Given $1\leq p<\infty,$ let $\ell^p$ be the set of $p$-Lebesgue space of sequences, that is, $a=(a(n))_{n\ge 0}\subset \CC$
$$
\Vert a \Vert_p:=\left(\sum_{n=0}^\infty \vert a(n)\vert^p\right)^{1\over p}<\infty,
$$
for $1\le p<\infty$ and $\Vert a\Vert_\infty:=\sup_n\vert a(n)\vert<\infty.$ The Lebesgue space $\ell^p$ is a module for the algebra $\ell^1$ and $a\ast b \in \ell^p$ where
$$
(a\ast b)(n)=\sum_{j=0}^n a(n-j)b(j), \qquad n\ge 0\qquad a\in \ell^p,\qquad b\in \ell^1,
$$
for $1\le p\le \infty$.

We write by $\delta_{n}$ for $n\in \ZZ$ the usual Dirac measure on $\ZZ$  defined by $\delta_n(m)= 1$ in the case $n=m$ and $\delta_n(m)= 0$ in other case. Note that $\delta_n \in \ell^p$ for $n\ge 0$ and $\delta_n\ast a$ is defined for any $n\in \ZZ$ in the following way
$$
(\delta_n\ast a)(m)=\begin{cases}  a(m-n), \quad & m\ge n,\\0,\quad & m<n.\end{cases}
$$
We write by $\delta_{-1}$ and $\delta_1$ the usual forward and backward shifts operator on $\ell^p$ space, i.e. $\delta_{-1}, \delta_1: \ell^p\to \ell^p$ given by $(\delta_{-1}\ast a)(n)=a(n+1)$ for $n\ge 0$ and $(\delta_{1}\ast a)(n)=a(n-1)$ for $n\ge 1 $ and $(\delta_{1}\ast a)(0)=0$.

 Let $(k^\alpha(n))_{n\ge 0}$ be the Ces\`aro numbers defined by
$$
k^\alpha(n)={\alpha(\alpha+1)\dots(\alpha+n-1)\over n!}={\Gamma(n+\alpha)\over \Gamma(\alpha)n!}, \qquad n\ge 0, \quad \alpha\in \R,
$$
see \cite[Vol. I. p.77]{[Zi]}. For $\alpha \in \R\setminus\{0,-1,-2, \dots\}$, we have that
$$
k^\alpha(n)= {n^{\alpha-1}\over \Gamma(\alpha)}\left(1+O({1\over n})\right), \qquad n \in \NN,
$$
see \cite[Vol. I. p.77 (1.18)]{[Zi]} and then $(k^\alpha(n))_{n\ge 0}\in \ell^p$ if and only if $\alpha<1-{1\over p}$ for $1\le p<\infty$ and $(k^\alpha(n))_{n\ge 0}\in \ell^\infty$ for $\alpha\le 1$. Also the kernel $(k^\alpha(n))_{n\ge 0}$ could be defined by the following generating formula
$$
\sum_{n=0}^\infty k^\alpha(n)z^n={1\over (1-z)^\alpha}, \qquad \vert z\vert <1.
$$

Given a set $I$, the characteristic function $\chi_I$ is defined by
$$
\chi_I(u)=\begin{cases}  0, \quad & u\not \in I,\\1, & u \in I.\end{cases}
$$

The functions $\Gamma$ and $\B$ is the usual Euler functions
$$
\Gamma(u)= \int_0^\infty t^{u -1}e^{-t} dt, \,\, \B (u,v)= \int_0^1(1-t)^{u-1}t^{v-1}dt=\int_0^\infty(1-e^{-s})^{u-1}e^{-sv}ds={\Gamma(u)\Gamma(v)\over \Gamma(u+v)}, $$
 for  $\Re u,\Re v>0.$

The $\Zeta$-transform of a sequence $a=(a(n))_{n\ge 0}$ is given by the Taylor series
$$
\Zeta(a)(z):=\sum_{n=0}^\infty a(n)z^n, \qquad z\in \DD:=\{z\in \CC \, :\, \vert z\vert < 1\}.
$$

The function ${\mathcal L}(f)$ is the usual Laplace transform of a measurable function $f:\R^+\to \CC$ defined as
$$
{\mathcal L}(f)(z):=\int_0^\infty f(t)e^{-zt}dt, \qquad z\in \CC^+.
$$

We write by $(X, \Vert \quad \Vert)$ a Banach space and ${\mathcal B}(X)$ is the set of linear and bounded operator on $X$. Given $(A, D(A))$ a closed operator on $X$, we denote by $\sigma_{point}(A)$ the point spectrum, $ \sigma(A)$ the  spectrum set and $(\lambda-A)^{-1}$ the resolvent operator for $\lambda \in \rho (A)$, the resolvent set. \label{conne}

\section{Modified Beta  function and Poisson-Charlier orthogonal polynomials}
\label{Charlier}

In this section we are interested in a modified beta function $\B_1(\cdot,\cdot)$ (which may expressed in terms of Kummer function) and in  a particular case of Poisson-Charlier polynomials $(p_n))_{n\ge 0}$.  These preliminary results will be useful in next sections.

\begin{definition}\label{Chat}{\rm We consider the modified Beta Euler function $\B_1$ defined by
$$
\B_1(u,v):= \int_0^1(1-t)^{u-1}t^{v-1}e^{-t}dt, \qquad u, v\in \CC^+.
$$
 }
\end{definition}
Note that
$$
\B_1(u,v)= e^{-1}\B(u,v)\,_1F_1(u, u+v,1), \qquad u, v\in \CC^+,
$$
where $\,_1F_1(a, b;z)$ is the Kummer function defined by
$$
\,_1F_1(a, c;z):={\Gamma(c)\over \Gamma(a)}\sum_{n=0}^\infty{\Gamma(a+n)\over \Gamma(c+n)}{z^n\over n!}= {\Gamma(c)\over \Gamma(a)\Gamma(c-a)}\int_0^1(1-t)^{c-a-1}t^{a-1}e^{zt}dt,
$$
for $\Re c >\Re a >0$, see definition and properties in \cite[Chapter VI]{MOS}. Some basic facts of this function are presented in the next proposition.

 \begin{proposition}\label{bb} \label{beta1} The function $\B_1$ verifies the following properties.
 \begin{itemize}
 \item[(i)] $ \vert \B_1(u,v)\vert\le    \B(\Re u,\Re v)$ for $u, v\in \CC^+$.
 \item[(ii)] $ e^{-1} \B(u,v)\le  \B_1(u,v)\le \B(u,v)$ for $u, v>0$.
 \item[(iii)] $\B_1(u+1,v+1)= v\B_1(u+1,v)-u\B_1(u,v+1)$ for $u, v\in \CC^+$.
 \item[(iv)] $\displaystyle{\B_1(u,v)= e^{-1}\sum_{n=0}^\infty {\B(u+n, v)\over n!}}$  for $u, v\in \CC^+$.
 \end{itemize}
 \end{proposition}

 \begin{proof} The proof of item (i) and (ii) are straightforward. To show (iii), we integrate by parts to get
 $$
 \B_1(u+1,v+1)= v\int_0^1(1-t)^{u}t^{v-1}e^{-t}dt-u\int_0^1(1-t)^{u-1}t^{v}e^{-t}dt=v\B_1(u+1,v)-u\B_1(u,v+1)
 $$
 for $u, v\in \CC^+$. Finally to check (iv), we have that
 \begin{eqnarray*}
 \B_1(u,v)&=& e^{-1}\B(u,v)\,_1F_1(u, u+v,1)=e^{-1}\B(u,v) {\Gamma(u+v)\over \Gamma(u)}\sum_{n=0}^\infty{\Gamma(u+n)\over \Gamma(u+v+n)}{z^n\over n!}\cr
 &=&e^{-1}\sum_{n=0}^\infty {\B(u+n, v)\over n!},
  \end{eqnarray*}
  for $u, v\in \CC^+$.\end{proof}

\begin{remark}{\rm We may check some particular values $\B_1(u,v)$ for $u, v\in \NN$, in particular
\begin{eqnarray*}
\B_1(1,1)&=&1-e^{-1},\cr
\B_1(n+1,1)&=& 1-n\B_1(n,1)=a(n)+(-1)^{n-1}(n-1)!e^{-1}, \quad n\ge 1, \cr
\B_1(1,n)&=&(n-1)!-e^{-1}\sum_{j=0}^{n-1}{(n-1)!\over j!}=(n-1)!-e^{-1}b(n-1)
\end{eqnarray*}
where  sequences $(a(n))_{n\ge 0}$ and $(b(n))_{n\ge 0}$ are defined by $a(0)=0$ and $a(n) = 1 - n a(n-1)·$ and are denoted by A182386  and A000522 in the well-known  On-line Encyclopedia of Integer Sequences (OEIS) by Sloane.

For $n, m\in \NN$, we iterate Proposition\ref{bb} (iii) to conclude that  $\B_1(n,m)= a_{n,m}+b_{n,m}e^{-1}$ with $a_{n,m}, b_{n,m} \in \ZZ$. It would be interested to identity the values of $a_{n,m}$ and $b_{n,m}$ in terms of $n$ and $m$. In Table 3 we present some values of $\B_1(n,m)$.

\medskip

\begin{table}[h]
\begin{center}
 \caption{ Values of $\B_1(n,m)$ for $1\le n,m \le 4$ }
\begin{tabular}{ |p{2cm}||p{2cm}|p{2.5cm}|p{2.5cm}|p{3cm}| }
 \hline
 $\B_1(n,m)$& $m=1$ &$m=2$ &$m=3$ &$m=4$\\
 \hline
$n=1$   & $1-e^{-1}$    &$1-2e^{-1}$&   $2-5e^{-1}$&$6-16e^{-1}$\\
$n=2$    & $e^{-1}$  &$-1+3e^{-1}$&$-4+11e^{-1}$&$-18+49e^{-1}$\\
$n=3$ &$1-2e^{-1}$ & $3-8e^{-1}$& $2(7-19e^{-1})$& $78-212e^{-1}$\\
$n=4$     &$-2+6e^{-1}$ &$-11+30e^{-1}$&$-64+174e^{-1}$& $6(-35+95e^{-1})$\\
 \hline
\end{tabular}
\end{center}
\end{table}
}
\end{remark}

In the rest of this section, we consider a particular case of Poisson-Charlier polynomials.

\begin{definition}\label{Chat}{\rm For $n\in \N_0$, we consider the sequence of polynomials $(p_n)_{n\ge 0}$ defined by $p_0(z)=1$ and
$$
p_n(z):=\sum_{j=0}^n{ n\choose j}z(z+1)\cdots(z+j-1)=\sum_{j=0}^n{ n\choose j}{\Gamma(j+z)\over \Gamma(z)}=n!\sum_{j=0}^n{ k^z(j)\over (n-j)!} , \qquad n\ge 1.
$$
 }
\end{definition}

The first values of this sequence are the following ones
\begin{eqnarray*}
p_0(z)&=&1,\cr
p_1(z)&=&z+1,\cr
p_2(z)&=&z^2+3z+1,\cr
p_3(z)&=&z^3+6z^2+8z+1,\cr
p_4(z)&=&z^4+10z^3+29z^2+24z+1,\cr
p_5(z)&=&z^5+15z^4+75z^3+145z^2+89z+1.
\end{eqnarray*}
Note that $\B_1(1,n)=(n-1)!-e^{-1}p_{n-1}(1)$ for $n\ge 1$. Some basic properties of these polynomials are given in the next theorem.

\begin{theorem}\label{Charlier} The sequence of polynomials $(p_n)_{n\ge 0}$ verifies the following properties
\begin{itemize}
\item[(i)] The polynomials $p_n$ are monic, of degree $n$ and $p_n(0)=1$ for $n\ge 0$.
\item[(ii)] $p_{n+1}(z)= (z+n+1)p_n(z)-np_{n-1}(z)$ for $n\ge 1$.
\item[(iii)] For $\vert w\vert <1$ and $z\in \CC$, we have that
$$
\sum_{n=0}^\infty p_n(z){w^n\over n!}=e^w(1-w)^{-z}.
$$
\end{itemize}
 Let $\nabla$ be the backward difference operator $\nabla(f)(z):=f(z)-f(z-1)$.
\begin{itemize}
\item[(iv)] $\nabla(p_n)=np_{n-1}$ for $n\ge 1$.
\item[(v)]$\nabla^2(p_n)(z)-(z+n)\nabla(p_n)(z)+np_{n}(z)=0$ for $n\ge 0$ and $z\in \CC$.

\end{itemize}

\end{theorem}

\begin{proof}The proof of item (i) is straightforward. To show (ii) note that,
\begin{eqnarray*}
&\,&(z+n+1)p_n(z)-np_{n-1}(z)=(z+n+1){\Gamma(z+n)\over \Gamma(z)}+\sum_{j=0}^{n-1}{ n\choose j}(z+j+1){\Gamma(j+z)\over \Gamma(z)}\cr
&\quad&=(z+n+1){\Gamma(z+n)\over \Gamma(z)}+\sum_{l=1}^{n}{ n\choose l-1}{\Gamma(l+z)\over \Gamma(z)}+\sum_{j=0}^{n-1}{ n\choose j}{\Gamma(j+z)\over \Gamma(z)}\cr
&\quad&={\Gamma(z+n+1)\over \Gamma(z)}+\sum_{l=1}^n{ n+1\choose l}{\Gamma(l+z)\over \Gamma(z)}+1=p_{n+1}(z).
\end{eqnarray*}
(iii) Take $\vert w\vert <1$ and $z\in \CC$. Then
\begin{eqnarray*}
\sum_{n=0}^\infty p_n(z){w^n\over n!}&=&\sum_{n=0}^\infty \sum_{j=0}^n{ n\choose j}{\Gamma(j+z)\over \Gamma(z)}{w^n\over n!}=\sum_{j=0}^\infty \left(\sum_{n=j}^\infty{w^n\over (n-j)!}\right){\Gamma(j+z)\over \Gamma(z)j!}\cr
&=&e^w\sum_{j=0}^\infty k^z(j)w^j=e^w(1-w)^{-z}.
\end{eqnarray*}

(iv) Note that
\begin{eqnarray*}
\nabla(p_n)(z)&=&\sum_{j=0}^n{ n\choose j}\left({\Gamma(j+z)\over \Gamma(z)}-{\Gamma(j+z-1)\over \Gamma(z-1)}\right)=\sum_{j=1}^n{ n\choose j}j{\Gamma(j-1+z)\over \Gamma(z)}\cr
&=& n\sum_{j=1}^n {(n-1)!\over (j-1)!(n-j)!}{\Gamma(j-1+z)\over \Gamma(z)}=np_{n-1}(z),
\end{eqnarray*}
for $n\ge 1$.  The item (v) is straightforward consequence from (iv) and (ii) and we conclude the proof.
\end{proof}

\begin{remark}{\rm
The classical monic Charlier  (or Poisson-Charlier) polinomials, $(C_n^{(a)})_{n\ge 0},$ are defined by the generating function
$$
e^{-aw}(1+w)^z= \sum_{n=0}^\infty C_n^{(a)}(z){w^n\over n!}, \qquad a\not=0.
$$
see \cite[Section VI.1]{Chihara}. Then $p_n(z)= (-1)^nC_n^{(1)}(-z)$ for $z\in \CC$. The recurrence relation in three terms given in Theorem \ref{Charlier} (ii) is a consequence of the general recurrence formula for Charlier polynomials,
$$
C_{n+1}^{(a)}(z)=(z-n-a)C_n^{(a)}(z)-anC_{n-1}^{(a)}(z),\qquad z\in \CC,
$$ see \cite[(VI.1.4)]{Chihara}.

Fixed $z\in \NN$, real positive sequences $(p_n(z))_{n>0}$ have appeared in big amount of different contexts. For example take $z=1$, the sequence $(p_n(1))_{n \ge 0}$ is the total number of ordered $k$-tuples for $0\le k\le n$ of distinct elements from an $n$-element set. The first values are  $1, 2,5,16,65, 326,\dots$ and it is denoted by A000522 in the  On-line Encyclopedia of Integer Sequences  by Sloane. For $z=2$,  we obtain the sequence denoted by A001339 which first values are $1, 3, 11, 49, 261, 1631, ...$

}
\end{remark}
Now we define polynomials \begin{equation}\label{qn}\displaystyle{q_n(z):={p_n(z)\over n!}}, \qquad  z\in \CC, \quad n\ge 0.\end{equation} In the following theorem includes additional properties for the sequence $(q_n)_{n\ge 0}$.

\begin{theorem}Let $(q_n)_{n\ge 0}$ polynomials give in \eqref{qn} and $1\le p\le \infty$. Then
\begin{itemize}
\item[(i)] The equality $(n+1)q_{n+1}(z)=(z+n+1)q_n(z)-q_{n-1}(z)$ holds for $z\in \CC$ and $n\ge 0$.
\item[(ii)] Then $q_n(z)= (k^z \ast \epsilon)(n)$ and  $(q_n(z))_{n\ge 0}\in \ell^p$ for $\Re z<1-{1\over p}$.
\end{itemize}
\end{theorem}

\begin{proof} The item (i) is a easy consequence from Theorem \ref{Charlier} (ii). To show (ii), take $z\in \CC$ and we have that
$$
{p_n(z)\over n!}= \sum_{j=0}^n{ k^z(j)\over (n-j)! }= (k^z\ast \epsilon)(n), \qquad n\ge 0,
$$
where $\epsilon=({1\over n!})_{n\ge 0}$. Note that $\epsilon \in \ell^1$ and $k^z\in \ell^p$ for $\Re z <1-{1\over p}$. We conclude $(k^z\ast \epsilon)  \in \ell^p$ for $\Re z<1-{1\over p}$.
\end{proof}

\section{Poisson transformations on Lebesgue spaces}

\setcounter{theorem}{0}
\setcounter{equation}{0}

Given $1\leq p<\infty,$ let $L^p(\mathbb{R}^+)$ be the set of Lebesgue $p$-integrable functions, that is, $f$ is a measurable function and $$||f||_p:=\left(\int_0^\infty |f(t)|^pdt\right)^{1/p}<\infty.$$
We write by $(L^\infty(\mathbb{R}^+), \Vert \qquad\Vert_\infty )$ the measurable Lebesgue function which are almost bounded on $\R^+$ and
$$
\Vert f\Vert_\infty:=\esssup\{\vert f(t)\vert \, ;\vert t\in \R^+\}.
$$
Remind that the Lebesgue space $L^p(\mathbb{R}^+)$ is a module for the algebra $L^1(\mathbb{R}^+)$ and $f\ast g \in L^p(\mathbb{R}^+)$ where
$$
(f\ast g)(t)=\int_0^tf(t-s)g(s)ds, \qquad t\in \R^+,\qquad f\in L^p(\mathbb{R}^+),\quad g\in L^1(\mathbb{R}^+),
$$
for $1\le p\le \infty$.

In this section we consider the following transform called Poisson transformation in some text, see for example \cite{Li}.

\begin{definition}\label{Poiss}{\rm  Take  $1\le p\le \infty$. We introduce the operator  ${\P}: L^p(\mathbb{R}^+)\to \ell^p$ defined by
\begin{equation}\label{poi}
{\P}(f)(n):=\int_0^\infty f(t) e^{-t}{t^n\over n!}dt, \qquad n \ge 0, \quad f\in L^p(\mathbb{R}^+).
\end{equation}
Duality, we consider the operator  ${\P}^*:  \ell^p \to L^p(\mathbb{R}^+)$ given by
\begin{equation}\label{poid}
{\P}^*(a)(t):=e^{-t}\sum_{n=0}^\infty a(n){t^n\over n!}, \qquad t\in \R^+, \quad a=(a(n))_{n\ge 0}\in \ell^p.
\end{equation}}
\end{definition}

\noindent{\bf Examples.} We write by $e_{\lambda}(t):=e^{-\lambda t}$ for $\Re \lambda>0$. Then
\begin{equation}\label{exp}
{\P}(e_\lambda)(n)= \int_0^\infty  e^{-(\lambda +1)t}{t^n\over n!}dt={1\over (1+\lambda)^n}, \qquad n\ge 0.
\end{equation}
Conversely, take $\lambda \in \CC$ with $\vert \lambda\vert <1$ and $a_\lambda(n):={ \lambda^n}$ for $n\ge 0$. Then
$$
{\P}^*(a_\lambda)(t):=e^{-t}\sum_{n=0}^\infty {(\lambda t)^n\over n!}=e^{-(1-\lambda)t}, \qquad t\ge 0.
$$

In fact, operators ${\P}$ and ${\P}^*$ may be defined in  larger domains. \begin{itemize}
\item[(i)] Take also $\displaystyle{j_{\alpha}(t):={t^{\alpha-1}\over \Gamma(\alpha)}}$ for $\alpha, t>0$. Then
$$
\P(j_\alpha)(n)= \int_0^\infty {t^{\alpha-1}\over \Gamma(\alpha)}e^{-t}{t^n\over n!}dt= {\Gamma(n+\alpha)\over \Gamma(\alpha)n!}= k^\alpha(n), \qquad n\ge 0.
$$

\item[(ii)] Conversely
$$\P^*(k^\alpha)(t)= {e^{-t}\over\Gamma(\alpha)}\sum_{n=0}^\infty {\Gamma(n+\alpha)\over (n!)^2}{t^n}=e^{-t}E^{\alpha}_{1,1}(t), \qquad t\ge 0,
$$
where $\displaystyle{E^\gamma_{\alpha, \beta}(z)=\sum_{n=0}^\infty {\Gamma(n+\gamma)\over \Gamma(\gamma)}{z^n\over n! \Gamma(\alpha n+\beta)}}$ is the entire function, called Prabhakar function, and  introduced in \cite{Pra}. Note that, for $\alpha \in \NN$, $\P^*(k^\alpha)$ is a  polynomial of degree $\alpha-1$.

\item[(iii)]Fix $a\in \CC$, we consider the sequence $P_a=(p_n(a))_{n\ge 0}$ where the sequence $(p_n)_{n\ge 0}$ is given in Definition \ref{Chat}. Then $\P^*(P_a)(t)=(1-t)^{-a}$ for $t<1$, see Theorem \ref{Charlier} (iii).
    \end{itemize}

\smallskip

In the next theorem we present more algebraic properties of operators $\P$ and $\P^*$.

\begin{theorem} \label{convoss} Take $1\le p\le \infty$ and operators $\P$ and $\P^*$ defined by (\ref{poi}) and  (\ref{poid}). Then
\begin{itemize}
\item[(i)] The map $\P: L^p(\mathbb{R}^+)\to \ell^p$ is a linear bounded operator and $\Vert \P\Vert = 1$.
\item[(ii)]  $\P(f\ast g)= \P(f)\ast \P(g)$ for $f\in L^p(\R^+)$ and $g\in L^1(\R^+)$.
\item[(iii)] For $1<p\le \infty$, the map $\P^*: \ell^p\to L^p(\mathbb{R}^+)$ is the adjoint operator of $\P$  and $\Vert \P^*\Vert = 1$ for $1\le p\le \infty$.
\item[(iv)] $\P^*(\delta_1\ast(a\ast b))= \P^*(a)\ast \P^\ast(b)$ for $a\in \ell^p$ and $b\in \ell^1$.
\item[(v)] Given $ a=(a(n))_{n\ge 0}\in \ell^p$, then
$$
\P(\P^\ast (a))(m)= {1\over 2^{m+1}}\sum_{n=0}^\infty {k^{m+1}(n)\over 2^{n}}a(n), \qquad m\ge 0,
$$
where $(k^{m+1}(n))_{n\ge 0} $ are  Ces\`aro numbers of order $m+1$.

\item[(vi)] Given $ f\in L^p(\R^+)$, then
$$
\P^*(\P(f))(t)=e^{-t}\int_0^\infty f(s)e^{-s}I_0(2\sqrt{ts})ds, \qquad t>0,
$$
where the Bessel function $I_0$ is given by $\displaystyle{I_0(z)=\sum_{n=0}^\infty {(z/2)^{2n}\over (n!)^2}}$, see for example \cite[p. 66]{MOS}.
\end{itemize}

\end{theorem}

\begin{proof} (i) First consider $p=1$ and $p=\infty$. Then it is straightforward to check that  $\Vert \P\Vert \le 1$ and the map $\P: L^p(\mathbb{R}^+)\to \ell^p$ is a linear and bounded operator for $p\in \{1, \infty\}$. By the Riesz–Thorin interpolation theorem, we conclude that   $\Vert \P\Vert \le 1$ for any $1\le p \le \infty$. Note take $\lambda >0 $ and consider functions $e_\lambda(t)=e^{-\lambda t}$. By (\ref{exp}), we obtain that
$$
{\Vert \P(e_\lambda)\Vert_p^p\over \Vert e_\lambda\Vert_p^p}= {(1+\lambda)^p \lambda p\over (1+\lambda)^p -1}\to 1, \qquad \lambda\to 0,
$$
and we conclude that $\Vert \P\Vert = 1$  for any $1\le p \le \infty$.

\noindent (ii) Take $f\in L^p(\R^+)$ and $g\in L^1(\R^+)$ then $ f\ast g\in L^p(\R^+)$ for $1\le p\le \infty.$ For   $n\ge 0$, we get that
\begin{eqnarray*}
\P(f\ast g)(n)&=& \int_0^\infty \int_0^t f(t-s)g(s)ds {t^n\over n!}e^{-t}dt= \int_0^\infty g(s)\int_s^\infty f(t-s) {t^n\over n!}e^{-t}dtds\cr
&=&\int_0^\infty g(s)e^{-s}\int_0^\infty f(u) {(u+s)^n\over n!}e^{-u}duds\cr
&=&\sum_{j=0}^n \int_0^\infty g(s)e^{-s}{s^j\over j!}ds \int_0^\infty f(u)e^{-u}{u^{n-j}\over (n-j)!}du= (\P(g)\ast \P(f))(n).
\end{eqnarray*}

\noindent (iii) Take  $1\le p< \infty$, $f\in L^p(\R^+)$ and $b \in \ell^{p'}$ with ${1\over p}+{1\over p'}=1$. Then
$$
\langle \P(f), b\rangle= \sum_{n=0}^\infty \P(f)(n)b(n)= \int_0^\infty f(t)e^{t} \sum_{n=0}^\infty{t^n\over n!}b(n)dt= \langle f, \P^*(b)\rangle.
$$
We conclude that the operator  $\P^*: \ell^p\to L^p(\mathbb{R}^+)$ is a bounded operator and verifies that $\Vert \P^*\Vert =1$ for $1<p\le \infty$. The case $p=1$ is shown directly.

\noindent{(iv)} Take  $a\in \ell^p$ and $b\in \ell^1$. Then
\begin{eqnarray*}
(\P^*(a)\ast \P^\ast(b))(t)&=&\int_0^t \P^*(a)(t-s) \P^*(b)(s)ds= \sum_{n,j=0}^\infty {a(n)\over n!}{b(j)\over j!}e^{-t}\int_0^t(t-s)^{n}s^jds\cr
&=& e^{-t}\sum_{n,j=0}^\infty {a(n)}{b(j)}{t^{n+j+1}\over (n+j+1)!}= e^{-t}\sum_{l=0}^\infty {t^{l+1}\over (l+1)!}\sum_{n=0}^{l}{a(n)}{b(l-n)}\cr
&=& e^{-t}\sum_{l=0}^\infty {t^{l+1}\over (l+1)!}(a\ast b)(l)= \P^*(\delta_1\ast(a\ast b))(t)
\end{eqnarray*}
for $t\ge 0$.

\noindent{(v)} Take $ a=(a(n))_{n\ge 0}\in \ell^p$. Then
\begin{eqnarray*}
\P(\P^\ast (a))(m)&=& \int_0^\infty{t^m\over m!}e^{-2t}\sum_{n=0}^\infty {t^n\over n!}a(n) dt= {1\over m!}\sum_{n=0}^\infty {a(n)\over n!}\int_0^\infty{t^{m+n}}e^{-2t}dt\cr
&=& {1\over m!}\sum_{n=0}^\infty {a(n)\over n!}{(m+n)!\over 2^{n+m+1}}= {1\over 2^{m+1}}\sum_{n=0}^\infty {k^{m+1}(n)\over 2^{n}}a(n), \cr
\end{eqnarray*}
for $m\ge 0$.

\noindent{(vi)} Take  $ f\in L^p(\R^+)$. Then
\begin{eqnarray*}
\P^*(\P(f))(t)&=&e^{-t}\sum_{n=0}^\infty{t^n\over n!}\int_0^\infty f(s)e^{-s}{s^n\over n!}ds =e^{-t}\int_0^\infty f(s)e^{-s}\sum_{n=0}^\infty{(st)^n\over (n!)^2}ds\cr
&=&e^{-t}\int_0^\infty f(s)e^{-s}I_0(2\sqrt{ts})ds, \qquad t>0,
\end{eqnarray*}
 and we conclude the proof.
\end{proof}

In the next proposition, we give the connection between Poisson transformations $\P$ and $\P^*$ and the Laplace and Zeta transform $\L$ and ${\mathcal Z}$ presented in the Introduction.

\begin{proposition}\label{conne} Let $\P$ and $\P^*$ operators given in Definition \ref{Poiss}.
\begin{itemize}
\item[(i)] Take $f\in L^p(\R^+)$ for $1\le p\le \infty$. Then $$\Zeta(\P (f))(z)= \L f(1-z), \qquad \vert z\vert<1.$$  In particular the operator $\P: L^p(\mathbb{R}^+)\to \ell^p$ in injective for $1\le p\le \infty$.
\item[(ii)] Take $a=(a(n))_{n\ge 0} \in \ell^p$. Then
$$
{\mathcal L}(\P^*(a))(z)={1\over 1+z}\Zeta(a)\left({1\over 1+z}\right), \qquad \Re z>0.
$$
In particular the operator $\P^*:  \ell^p\to L^p(\mathbb{R}^+)$ in injective for $1\le p\le \infty$.
\item[(iii)] Both operators $\P$ and $\P^*$ are of dense range.

\end{itemize}

\end{proposition}
\begin{proof} (i) Take   $f\in L^p(\mathbb{R}^+)$. Then we get that
$$
\Zeta(\P (f))(z)= \sum_{n=0}^\infty z^n\int_0^\infty  f(t){t^n\over n!}dt= \int_0^\infty f(t) e^{-(1-z)t}dt={\mathcal L}(f)(1-z),
$$
for $\vert z\vert<1$. In the case that $\P(f)=0$, we conclude that $f=0$ by the injectivity of the Laplace transform.

\noindent (ii) Take $a=(a(n))_{n\ge 0} \in \ell^p$. Then
\begin{eqnarray*}
{\mathcal L}(\P^*(a))(z)&=&\int_0^\infty e^{-(z+1)t} \sum_{n=0}^\infty {t^n\over n!}a(n)dt= \sum_{n=0}^\infty {a(n)\over n!}\int_0^\infty t^ne^{-(z+1)t}dt\cr
&=& \sum_{n=0}^\infty {a(n)\over (1+z)^{n+1}}= {1\over 1+z}\Zeta(a)\left({1\over 1+z}\right),
\end{eqnarray*}
 for  $\Re z>0$. In the case that $\P^*(a)=0$, we conclude that $a=0$ by the injectivity of the Zeta transform.

 The item (iii) is a straight consequence of the injectivity of $\P$ and $\P^*$.
\end{proof}

\begin{remark}{\rm Note that both operators ${\mathcal P}$ and ${\mathcal P}^*$ are not invertible for $1\le p\le \infty$. Moreover the range of ${\mathcal P}^*$ is contained in the set of analytic functions. Now, suppose that  ${\mathcal P}$ was invertible, there will exist $f\in L^p(\R^+)$ such that $\mathcal P(f)=\delta_0$. By Proposition \ref{conne}(i)
$$
1= \L f(1-z), \qquad z \in D(0,1),
$$
and this follows  a contradiction with $f\in L^p(\R^+)$.
}
\end{remark}

\section{Some $C_0$-semigroups on  $\ell^p$ and $L^p(\R^+)$}
\label{sec:discrete}

In this section we consider two different families of $C_0$-semigroups on $\ell^p$: convolution semigroups ( $(e^{z\Delta})_{z\in \CC}$ and $(e^{z\nabla})_{z\in \CC}$) and Koopman semigroups ($(\frak{T}_p(t))_{t>0}$ and
$(\frak{S}_p(t))_{t>0}$). We also present nice connections with $C_0$-semigroups on $L^p(\RR^+)$ in Theorem \ref{eliz} and \ref{intert2}.

\subsection{Convolution semigroups on  $\ell^p$}\label{savi}

The forward difference operator $\Delta(a)(n):= ((\delta_{-1}-\delta_0)\ast a)(n)=a(n+1)-a(n)$ is a bounded operator on $\ell^p$, and $\Vert \Delta\Vert=2$. In fact,  the symbol $(\delta_{-1}-\delta_0)$ generates a entire semigroup of elements on $\ell^p(\ZZ)$ where
$$
e^{z(\delta_{-1}-\delta_0)}(n)=e^{-z}{z^{-n}\over (-n)!}\chi_{-\NN_0}(n), \qquad n\in \ZZ, \quad z\in \CC,
$$
and then also defined a entire $C_0$-semigroup of operators on $\ell^p,$
$(e^{z\Delta})_{z\in \CC}\subset {\mathcal B}(\ell^p),$ where
$$
(e^{z\Delta }a)(n)= e^{-z}\sum_{j\ge 0}^\infty a(j+n){z^j\over j!}, \qquad a\in \ell^p ,
$$
for $1\le p \le \infty$. Its infinitesimal generator is $\Delta$ and $\Vert e^{t\Delta }\Vert = 1$ for $t\ge 0$, see similar results on $\ell^p(\ZZ)$ in \cite[Theorem 3.2]{GLM}. Note that
$$
(e^{t\Delta }a)(n)=\P^*(\delta_{-n}\ast a)(t), \qquad t>0, \qquad a\in \ell^p.
$$

The classical backward difference operator $\nabla(a)(n):= ((\delta_{1}-\delta_0)\ast a)(n)=a(n-1)-a(n)$ for $n\ge 1$ and $\nabla(a)(0):=-a(0)$ is also a bounded operator on $\ell^p$, and $\Vert \nabla\Vert=2$. In this case  the element $(\delta_{1}-\delta_0)$ generates an entire semigroup on $\ell^p$ where
$$
e^{z(\delta_{1}-\delta_0)}(n)=e^{-z}{z^{n}\over n!}, \qquad n\ge 0, \quad z\in \CC.
$$
The operator $\nabla$ is the infinitesimal generator of  the entire $C_0$-semigroup on $\ell^p,$
$(e^{z\nabla})_{z\in \CC}\subset {\mathcal B}(\ell^p),$ where
$$
(e^{z\nabla}a)(n)= (e^{z(\delta_{1}-\delta_0)}\ast a)(n)  =  e^{-z}\sum_{j\ge 0}^n a(j){z^{n-j}\over (n-j)!}, \qquad a\in \ell^p ,
$$
for $1\le p \le \infty$ and  $\Vert e^{t\nabla}\Vert = 1$ for $t\ge 0$, see  \cite[Theorem 3.3]{GLM}. It is straightforward to check that $\nabla$ is the adjoint operator of $\Delta$, i.e,  $(\Delta)^*=\nabla $ and $(e^{z\Delta})^*=e^{z\nabla}$ on $\ell^{p'}$ where ${1\over p}+{1\over p'}=1 $ for $1\leq p< \infty$.

\begin{remark} {\rm For $1\le p\le \infty$, note that
$$
\sigma_{{\mathcal B}(\ell^p)}(\Delta)=\sigma_{{\mathcal B}(\ell^p)}(\nabla)=\{z\in \CC\, :\, \vert z-1\vert \le 1\},
$$
as a nice consequence of Wiener’s lemma, see for example \cite{Gro}.}
\end{remark}

The usual left translation $C_0$-semigroup $(T_{left}(t))_{t>0}$ on $L^p(\R^+)$ defined by
$$
T_{left}(t)f(s):=f(s+t), \quad s,t>0, \qquad f\in L^p(\R^+).
$$
In fact, it is a $C_0$-semigroups of contractions and the infinitesimal generator is the usual derivation $({d\over ds}, D({d\over ds}))$,
$$
{d\over ds}f(s):=f'(s), \qquad f\in D({d\over ds})=\{f \in L^p(\R^+)\,\, :\, \, f' \in L^p(\R^+)\},
$$
 for $1\le p<\infty$, see \cite[Section I.4.16]{En-Na-00}. The right translation $C_0$-semigroup $(T_{right}(t))_{t>0}$ on $L^p(\R^+)$,  $1\le p<\infty$, defined by
$$
T_{right}(t)f(s):= \begin{cases}  f(s-t), \quad &s\ge t,\\0,\quad &0\le s<t.\end{cases}
$$
Also it is a $C_0$-semigroups of contractions and the infinitesimal generator is the  derivation $({d^0\over ds}, D({d^0\over ds}))$,
$$
{d^0\over ds}f(s):=-f'(s), \qquad f\in D({d^0\over ds})=\{f \in L^p(\R^+)\,\, : \, \, f' \in L^p(\R^+), \hbox{ and }f(0)=0\}.
$$
For $1 < p < \infty$, it is straightforward to prove that  the semigroups $(T_{left}(t))_{t>0}$ and $(T_{right}(t))_{t>0}$ are adjoint
to each other, i.e., $(T_{left}(t))^*$
coincides with $T_{right}(t)$ on $L^{p'}(\R^+)$ where $1/p + 1/p'= 1$ and $t>0$.

In the following theorem, we show that $C_0$-semigroups $(T_{left}(t))_{t>0}$, $(T_{right}(t))_{t>0}$,  $(e^{t\Delta})_{t\geq 0}$ and $(e^{z\nabla})_{t\geq 0}$ intertwine with $\P$ and $\P^*$. In fact, this property holds for its infinitesimal generators.

\begin{theorem}\label{eliz} Let $(T_{left}(t))_{t>0}$ and $(T_{right}(t))_{t>0}$ be the translations $C_0$ semigroups defined on $L^p(\R^+)$ and the $C_0$-semigroups $(e^{t\Delta})_{t\geq 0}$ and $(e^{z\nabla})_{t\geq 0}$ on $\ell^p$ for $1\le p\le \infty$. The following equalities hold.
\begin{itemize}
\item[(i)] $\nabla(\P(f))=-\P(f')$ for $f\in D({d^0\over ds})$.
\item[(ii)] $ e^{t\nabla}\circ \P= \P\circ T_{right}(t)$ for  $t\ge 0$.
\item[(iii)] $\P^*(\Delta (a))= (\P^*(a))'$ for $a\in \ell^p$.
\item[(iv)]$\P^*(e^{t\Delta})=\P^*(T_{left}(t))$ for  $t\ge 0$.
\end{itemize}
\end{theorem}
\begin{proof} (i) Take $f\in D({d^0\over ds})$. As $f(0)=0$, $-\P(f')(0)= \int_0^\infty e^{-t}f(t)dt= \P(f)(0)= \nabla(\P(f))(0).$ For $n\ge 1$,
$$
-\P(f')(n)= \int_0^\infty ({t^n\over n!}e^{-t})'f(t)dt= \P(f)(n-1)-\P(f)(n)= \nabla(\P(f))(n).
$$
\noindent (ii). Take $t>0$ and $f\in L^p(\R)$. For $n\ge 0$, we obtain
\begin{eqnarray*}
e^{t\nabla}( \P(f))(n)&=&e^{-t}\sum_{j=0}^n \int_0^\infty{s^j\over j!}e^{-s}f(s)ds{t^{n-j}\over (n-j)!}=  \int_0^\infty f(s)e^{-(t+s)}\sum_{j=0}^n{s^j\over j!}{t^{n-j}\over (n-j)!}ds\cr
&=& \int_0^\infty f(s)e^{-(t+s)}{(s+t)^n\over n!}ds= \P( T_{right}(t)f)(n).
\end{eqnarray*}
Items (iii) and (iv) are shown in a similar way. Alternatively we may use also duality to shown them.
\end{proof}

\subsection{ Some Koopman semigroups on  $\ell^p$}\label{savi2}
Now we consider the one-parameter families of operators $(\frak{T}(t))_{t\geq 0}$ and $(\frak{S}(t))_{t\geq 0}$ acting on  $\ell^p$ with $1\le p\le \infty$, where
\begin{eqnarray}\label{semis}
\frak{T}_p(t)a(n)&:=&e^{-{t\over p}}\displaystyle\sum_{j=0}^n{n \choose j}e^{-tj}(1-e^{-t})^{n-j}a(j),\cr
\frak{S}_p(t)a(n)&:=&e^{-t(1-{1\over p})}e^{-tn}\displaystyle\sum_{j=n}^{\infty}{j\choose n}(1-e^{-t})^{j-n}a(j),
\end{eqnarray}
for $n\in \N_0$ and $a\in \ell^p$. Respectively its infinitesimal generators are given by
\begin{eqnarray*}
 \frak{A}_p(a)(n)&:=&n (a(n-1)-a(n))-{1\over p}a(n)= n\nabla(a)(n)-{1\over p}a(n), \, n\ge 1,\cr\frak{A}_p(a)(0)&:=&- {1\over p}a(0); \cr
\frak{B}_p(a)(n)&:=&(n+1)a(n+1)-na(n)-(1-{1\over p})a(n)= (n+1)\Delta(a)(n)+{1\over p}a(n), \, n\ge 0,
\end{eqnarray*}
for $a\in D(\frak{A}_p)=D(\frak{B}_p)=\{b\in\ell^p\,\, ; (n+1)\Delta (b)\in \ell^p\}$. In fact, both families  $(\frak{T}_p(t))_{t\geq 0}$ and $(\frak{S}_p(t))_{t\geq 0}$ are dual contractions $C_0$-semigroups in $\ell^p$, $\Vert \frak{T}_p(t)\Vert, \Vert \frak{S}_p(t)\Vert\le 1 $ for $1\le p\le \infty$. See more details in \cite[Theorem 6.1, Proposition 6.3]{AM2018}. Moreover the operator $\frak{B}_{p'}$ is the adjoint operator of $\frak{A}_p$, i.e,  $(\frak{A}_p)^*=\frak{B}_{p'}$  on $\ell^{p'}$ where ${1\over p}+{1\over p'}=1 $ for $1\leq p < \infty$.

For $1\le p<\infty$, $\sigma_{point}(\frak{A}_p) =\emptyset$;  $\sigma_{point}(\frak{B}_p)=\CC_{-}$ for $1<p<\infty$ and  $\sigma_{point}(\frak{B}_p)=\CC_{-}\cup\{0\}$. Note that $\sigma(\frak{A}_p)= \sigma(\frak{B}_p) =\overline{\CC_-} $ for $1\le p<\infty$, see \cite[Proposition 6.4]{AM2018}. In the next theorem, we obtain the expression of $(\lambda-\frak{A}_p)^{-1}$ and $(\lambda-\frak{B}_p)^{-1}$ for $\Re \lambda>0$.
\begin{proposition} Take $\lambda \in \CC$ such that $\Re \lambda>0$ and $a=(a(n))_{n\ge 0}\in \ell^p$ for $1\le p<\infty$.
\begin{itemize}
\item[(i)] Then $\lambda \in \rho(\frak{A}_p)$ and
$$( \lambda-\frak{A}_p)^{-1}a(n)= {n!\over \Gamma(n+\lambda+{1\over p}+1)}\sum_{j=0}^n{\Gamma(\lambda+{1\over p}+j)\over j!}a(j), \qquad n\ge 0.
$$
\item[(ii)] Then $\lambda \in \rho(\frak{B}_p)$ and
$$
(\lambda- {\frak B}_{p})^{-1}a(n)= {\Gamma(n+\lambda+1-{1\over p})\over n!}\sum_{j=n}^\infty{j!\over \Gamma(\lambda+2-{1\over p}+j)}a(j),\qquad n\ge 0.
$$
\end{itemize}
\end{proposition}

\begin{proof} (i) Take $\lambda \in \CC$ such that $\Re \lambda>0$. Then
\begin{eqnarray*}
 (\lambda-\frak{A}_p)^{-1}a(n)&=&\int_0^\infty e^{-\lambda t} \frak{T}_p(t)a(n)dt= \displaystyle\sum_{j=0}^n{n \choose j}a(j)\int_0^\infty e^{-(\lambda +{1\over p}+j)t}(1-e^{-t})^{n-j}dt\cr
&=&{n!\over \Gamma(n+\lambda+{1\over p}+1)}\sum_{j=0}^n{\Gamma(\lambda+{1\over p}+j)\over j!}a(j),
\end{eqnarray*}
for $n\ge 0.$  (ii) Similarly, we also obtain that
\begin{eqnarray*}
 (\lambda-\frak{B}_p)^{-1}a(n)&=&\int_0^\infty e^{-\lambda t} \frak{S}_p(t)a(n)dt= \displaystyle\sum_{j=n}^{\infty}{j\choose n}a(j)\int_0^\infty e^{-t(\lambda+n +1-{1\over p})}(1-e^{-t})^{j-n}dt\cr
 &=&{\Gamma(n+\lambda+1-{1\over p})\over n!}\sum_{j=n}^\infty{j!\over \Gamma(\lambda+2-{1\over p}+j)}a(j),
\end{eqnarray*}
for $n\ge 0$ and we conclude the proof.
\end{proof}

\begin{remark}\label{notcommute} {\rm Note that operator $\frak{A}_p$ and $\Delta $ do not commute. In fact, for $n\ge 1$, we have that
\begin{eqnarray*}
\frak{A}_p(\Delta (a))(n)&=& n (2a(n)-a(n+1)-a(n-1))-{1\over p}(\Delta (a))(n),\cr
\Delta(\frak{A}_p (a))(n)&=& (2n+1)a(n)-(n+1)a(n+1)-na(n-1)-{1\over p}(\Delta (a))(n).\cr
\end{eqnarray*}
Similarly, we may check that   $\frak{B}_p$ and $\nabla $ do not commute.}
\end{remark}

On $L^p(\R^+)$, we consider the $C_0$-group of isometries, $(T_p(t))_{t\ge 0}$, defined by
\begin{equation}\label{eq:semigroups}
T_p(t)f(r):= e^{-{t\over p}}f(e^{-t}r), \qquad r\ge 0, \quad t\in \R.
\end{equation}
Its infinitesimal generator  $(\Lambda_p, D(\Lambda_p))$ is given by
$$
\Lambda_p (f)(r)= -rf'(r)-{1\over p}f(r), \qquad r\ge 0,
$$
and $D(\Lambda_p)=\{f\in L^p(\R^+)\,\, \vert \,\, rf'(r)\in L^p(\R^+)\}$, see \cite[Theorem 2.5]{LMPS}. We denote by $(T^+_p(t))_{t>0}$ and $(T^-_p(t))_{t>0}$ the $C_0$-semigroups defined by $T^+_p(t):= T_p(t)$ and $T^-_p(t):= T_p(-t)$ for $t\ge 0$;   $\Lambda_p$ and $-\Lambda_p$ are its infinitesimal generator, respectively, and they are adjoint (or dual) operators on suitable spaces (\cite[Proposition 2.7]{LMPS}).

In the following theorem, we show that $C_0$-semigroups $(T^+_p(t))_{t>0}$, $(T^-_p(t))_{t>0}$,  $(\frak{T}_p(t))_{t\geq 0}$ and $(\frak{S}_p(t))_{t\geq 0}$ intertwine with $\P$ and $\P^*$. The proof is similar to \cite[Theorem 6.3]{AMM} and we include here to avoid the lack of completeness.

\begin{theorem}\label{intert2} Take  the one-parameter families  $(T_p(t))_{t\in \R}$  defined on $L^p(\R^+)$ in (\ref{eq:semigroups}) and $(\frak{T}_p(t))_{t\geq 0}$ and $(\frak{S}_p(t))_{t\geq 0}$ defined on $\ell^p$ in (\ref{semis}) for $1\le p\le \infty$. Then
\begin{itemize}
\item[(i)] $\frak{B}_p(\P(f))= \P(-\Lambda_p(f))$ for $f\in D(\Lambda_p)$.
    \item[(ii)] $\frak{S}_p(t)\circ \P= \P\circ T^-_p(t)$ for $t>0$.
    \item[(iii)] $\P^*(\frak{A}_p(a))= \Lambda_p(\P^*(a))$ for $a\in D(\frak{A}_p)$.
     \item[(iv)] $ \P^*\circ\frak{T}_p(t)=  T^+_p(t)\circ \P^*$ for $t>0$.
    \end{itemize}

\end{theorem}
\begin{proof} (i) Take $f\in D(\Lambda_p)$ and $n\ge 0$. We integrate by parts to obtain
\begin{eqnarray*}
-\P(\Lambda_p(f))(n)&=& \int_0^\infty{t^{n+1}\over n!}e^{-t}f'(t)dt+{1\over p}\P(f)(n)\cr
&=&(n+1)\left(\P(f)(n+1)-\P(f)(n)\right)+{1\over p}\P(f)(n)\cr
&=&\left((n+1)\Delta+{1\over p}\right)\P(f)(n)=\frak{B}_p(\P(f))(n).
\end{eqnarray*}

\noindent(ii) Take $t>0$ and $f\in L^p(\R^+)$ with $1\le p\le\infty$. For $n\in \N_0$, we obtain that
\begin{eqnarray*}
(\frak{S}_p(t)\circ \P)(f)(n)&=&e^{-t(1-{1\over p})}e^{-tn}\displaystyle\sum_{j=n}^{\infty}{j\choose n}(1-e^{-t})^{j-n} \int_0^\infty{s^j\over j!}e^{-s}f(s)\,ds\cr
&=&e^{-t(1-{1\over p})}e^{-tn}\int_0^\infty e^{-s}f(s)\displaystyle\sum_{j=n}^{\infty}{1\over n!(j-n)!}(1-e^{-t})^{j-n}s^j\,ds\cr
&=&e^{-t(1-{1\over p})}e^{-tn}\int_0^\infty e^{-s}f(s){s^n\over n!}\displaystyle\sum_{m=0}^{\infty}{(s(1-e^{-t}))^{m}\over m!}\,ds\cr
&=&e^{-t(1-{1\over p})}e^{-tn}\int_0^\infty e^{-s}f(s){s^n\over n!}e^{s(1-e^{-t})}\,ds \cr
&=&e^{t\over p}\int_0^{e^{-t}}f(e^tu){u^n\over n!}e^{-u}\,du=  (\P\circ T^-_p(t))(f)(n).
\end{eqnarray*}

\noindent (iii) Take $a\in D(\frak{A}_p)$ and $t>0$. Then
\begin{eqnarray*}
\Lambda_p(\P^*(a))(t)&=& -t\left(\sum_{n=0}^\infty a(n)e^{-t}{t^n\over n!}\right)'-{1\over p }\P^*(a)(t)\cr
&=&\sum_{n=1}^\infty n (a(n-1)-a(n)){t^n\over n!}-{1\over p }\P^*(a)(t)=\P^*(\frak{A}_p(a))(t).
\end{eqnarray*}

    \noindent (iv) Now take  $t>0$ and $a\in \ell^p$ with $1\le p\le\infty$.  For $0<s$, we get

\begin{eqnarray*}
(\P^*\circ\frak{T}_p(t))(a)(s)&=&e^{-{t\over p}}\sum_{n=0}^\infty e^{-s}{s^n\over n!}\displaystyle\sum_{j=0}^{n}{n\choose j}e^{-tj}(1-e^{-t})^{n-j}a(j) \cr
&=&e^{-{t\over p}}e^{-s}\sum_{j=0}^\infty {e^{-tj}\over j!}a(j)\displaystyle\sum_{n=j}^{\infty}{s^n(1-e^{-t})^{n-j}\over (n-j)!}\cr
&=&e^{-{t\over p}}e^{-s}\sum_{j=0}^\infty {(se^{-t})^j\over j!}a(j)\displaystyle\sum_{m=0}^{\infty}{(s(1-e^{-t}))^{m}\over m!}\cr
&=&e^{-{t\over p}}\sum_{j=0}^\infty {(se^{-t})^j\over j!}e^{-se^{-t}}a(j)=  (T^+_p(t)\circ \P^*)(a)(s),
\end{eqnarray*}
and we conclude the proof.
\end{proof}

A direct consequence of Theorem  \ref{eliz} (i) and (iii) and Theorem \ref{intert2}  (i) and (iii) is the following corollary.
\begin{corollary} Take  $1\le p\le \infty$. Then
\begin{itemize}
\item[(i)] $(\nabla+\frak{B}_p)(\P(f))= \P((-\Lambda_p+{d^0\over d{s}})(f))$ for $f\in D(\Lambda_p)\cap D({d^0\over ds})$.
\item[(iii)] $\P^*((\frak{A}_p+\Delta)(a))= (\Lambda_p+{d\over ds})(\P^*(a))$ for $a\in D(\frak{A}_p)$.
\end{itemize}
\end{corollary}

\begin{remark} {\rm In the case that $1\le p <\infty$, we may prove Theorem \ref{eliz} (iii) and (iv) by duality. For example,  take $t>0$  and we get
$$
 \P^*\circ\frak{T}_p(t)= (\frak{S}_{p'}(t)\circ \P)^*= (\P\circ T^{-}_{p'}(t))^*=  T^+_p(t)\circ \P^*,
$$
where $\displaystyle{{1\over p}+{1\over p'}=1 }$.}
\end{remark}

r

\section{Perturbations of  Koopman semigroups on  $\ell^p$ } \label{savi3}

 In this main section we introduce two Koopman semigroups $(\frak{T}_{\Delta,p}(t))_{t\ge 0}$ and  $(\frak{S}_{\nabla,p}(t))_{t\ge 0}$ on $\ell^p$ whose infinitesimal generators are $\Delta+{\frak A}_{ p}$ and $\nabla +{\frak B}_{ p}$, bounded perturbations of ${\frak A}_{ p}$ and ${\frak B}_{ p}$. These $C_0$-semigroups and other two Koopman semigroups $(R_{p}(t))_{t\ge 0}$ and $(S_{p}(t))_{t\ge 0}$ on $L^p(\RR^)$ intertwine with operators $\mathcal P$ and $\mathcal P^*$, see Theorem \ref{conmutante}.

\begin{definition}\label{deff} Take $a\in \ell^p$ for $1\le p\le \infty$ and $t>0$. We define the operators
\begin{eqnarray*}\label{semis}
\frak{T}_{\Delta,p}(t)a(l)&:=&e^{-({t\over p}+1-e^{-t})}\displaystyle\sum_{j=0}^l{l \choose j}(1-e^{-t})^{l-j}e^{-tj}\sum_{n=j}^\infty{(1-e^{-t})^{n-j}\over (n-j)!}a(n),\cr
\frak{S}_{\nabla,p}(t)a(l)&:=&e^{-(t(1-{1\over p})+1-e^{-t})}\displaystyle\sum_{j=0}^l{(1-e^{-t})^{l-j}\over (l-j)!}e^{-tj}\sum_{n=j}^\infty{n \choose j}(1-e^{-t})^{n-j}a(n),
\end{eqnarray*}
for $l\ge 0.$
\end{definition}

\begin{theorem}\label{norms} Let $(\frak{T}_{\Delta,p}(t))_{t>0}$ and $(\frak{S}_{\nabla,p}(t))_{t>0}$  operators given in Definition \ref{deff}. Then
\begin{itemize}
\item[(i)] $\Vert  \frak{T}_{\Delta,p}(t)\Vert \le 1$  for $t>0$ and $1\le p\le \infty$.
\item[(ii)] $\Vert  \frak{S}_{\nabla,p}(t)\Vert \le 1$  for $t>0$ and $1\le p\le \infty$.
\item[(iii)] $(\frak{T}_{\Delta,p}(t))^*=  \frak{S}_{\nabla,p'}(t)$ and  $(\frak{S}_{\nabla,p}(t))^*=\frak{T}_{\Delta,p'}(t)$ on $\ell^p$ for $1\le p<\infty$ and ${{1\over p}+{1\over p'}=1 }$.
\end{itemize}

\end{theorem}

\begin{proof}(i) Firstly we check that $\frak{T}_{\Delta,1}(t):\ell^1\to \ell^1$  and $\Vert \frak{T}_{\Delta, 1}(t)\Vert \le 1$ for $t>0$. Note that

\begin{eqnarray*}
\sum_{l=0}^\infty \vert \frak{T}_{\Delta,1}(t)a(l)\vert &\le &\sum_{l=0}^\infty e^{-({t}+1-e^{-t})}\sum_{j=0}^l{l \choose j}(1-e^{-t})^{l-j}e^{-tj}\sum_{n=j}^\infty{(1-e^{-t})^{n-j}\over (n-j)!}\vert a(n)\vert\cr
&=&e^{-({t}+1-e^{-t})}\sum_{j=0}^\infty  e^{-tj}\sum_{l=j}^\infty
{l \choose j}(1-e^{-t})^{l-j}\sum_{n=j}^\infty{(1-e^{-t})^{n-j}\over (n-j)!}\vert a(n)\vert\cr
&=&e^{-({t}+1-e^{-t})}\sum_{n=0}^\infty \vert a(n)\vert \sum_{j=0}^n e^{-tj}{(1-e^{-t})^{n-j}\over (n-j)!}\sum_{l=j}^\infty{l \choose j}(1-e^{-t})^{l-j}.
\end{eqnarray*}
As
$$
\sum_{l=j}^\infty{l \choose j}(1-e^{-t})^{l-j}= \sum_{p=0}^\infty {p+j\choose j}(1-e^{-t})^p=\sum_{p=0}^\infty k^{j+1}(p)(1-e^{-t})^p=e^{t(j+1)}
$$
we have that
$$
\sum_{l=0}^\infty \vert \frak{T}_\Delta(t)a(l)\vert \le e^{-(1-e^{-t})}\sum_{n=0}^\infty \vert a(n)\vert \sum_{j=0}^n {(1-e^{-t})^{n-j}\over (n-j)!}\le \sum_{n=0}^\infty \vert a(n)\vert.
$$
Now, for $p=\infty$, and $a=(a(n))_{n\ge 0}\in \ell^\infty$,  we have that
\begin{eqnarray*}
\vert \frak{T}_{\Delta,\infty}(t)a(l)\vert &\le &e^{-(1-e^{-t})}\displaystyle\sum_{j=0}^l{l \choose j}(1-e^{-t})^{l-j}e^{-tj}\sum_{n=j}^\infty{(1-e^{-t})^{n-j}\over (n-j)!}\Vert a\Vert_\infty \cr
&\le & \Vert a\Vert_\infty \sum_{j=0}^l{l \choose j}(1-e^{-t})^{l-j}e^{-tj}= \Vert a\Vert_\infty,
\end{eqnarray*}
for $l\ge 0$. By the Riesz-Thorin interpolation theorem, we conclude that $\frak{T}_{\Delta,p}(t):\ell^p\to \ell^p$ and  $\Vert  \frak{T}_{\Delta,p}(t)\Vert \le 1$  for $t>0$ and $1\le p\le \infty$.

(ii) Similarly we check first $\frak{S}_{\nabla,1}(t):\ell^1\to \ell^1$  and $\Vert \frak{S}_{\nabla, 1}(t)\Vert \le 1$ for $t>0$.  Note that
\begin{eqnarray*}
\sum_{l=0}^\infty \vert \frak{S}_{\nabla,1}(t)a(l)\vert&\le&e^{-(1-e^{-t})}\sum_{l=0}^\infty\sum_{j=0}^l{(1-e^{-t})^{l-j}\over (l-j)!}e^{-tj}\sum_{n=j}^\infty{n \choose j}(1-e^{-t})^{n-j}\vert a(n)\vert,\cr
&=&e^{-(1-e^{-t})}\sum_{j=0}^\infty e^{-tj}\sum_{l=j}^\infty{(1-e^{-t})^{l-j}\over (l-j)!}\sum_{n=j}^\infty{n \choose j}(1-e^{-t})^{n-j}\vert a(n)\vert,\cr
&=&e^{-(1-e^{-t})}\sum_{n=0}^\infty \vert a(n)\vert \sum_{j=0}^n{n \choose j}e^{-tj}(1-e^{-t})^{n-j}\sum_{k=0}^\infty{(1-e^{-t})^{k}\over k!},\cr
&=&\sum_{n=0}^\infty \vert a(n)\vert \sum_{j=0}^n{n \choose j}e^{-tj}(1-e^{-t})^{n-j}=\sum_{n=0}^\infty \vert a(n)\vert. \cr
\end{eqnarray*}
Now take $p=\infty$ and in this case
\begin{eqnarray*}
\vert\frak{S}_{\nabla,\infty}(t)a(l)\vert &\le &e^{-(t+1-e^{-t})}\displaystyle\sum_{j=0}^l{(1-e^{-t})^{l-j}\over (l-j)!}e^{-tj}\sum_{n=j}^\infty{n \choose j}(1-e^{-t})^{n-j}\Vert a\Vert_{\infty}\cr
&= &e^{-(t+1-e^{-t})}\displaystyle\sum_{j=0}^l{(1-e^{-t})^{l-j}\over (l-j)!}e^{-tj}\sum_{k=0}^\infty{k+j \choose j}(1-e^{-t})^{k}\Vert a\Vert_{\infty}\cr
&= &e^{-(1-e^{-t})}\displaystyle\sum_{j=0}^l{(1-e^{-t})^{l-j}\over (l-j)!}\Vert a\Vert_{\infty}\le \Vert a\Vert_{\infty}, \cr
\end{eqnarray*}
for $l\ge 0$. By the  Riesz-Thorin interpolation theorem, we aslso conclude that $\frak{S}_{\nabla,p}(t):\ell^p\to \ell^p$ and  $\Vert  \frak{S}_{\nabla,p}(t)\Vert \le 1$  for $t>0$ and $1\le p\le \infty$.

(iii) Take $a\in \ell^p$ and $b\in \ell^{p'}$ for  $1\le p<\infty$ and ${{1\over p}+{1\over p'}=1 }$.Then
\begin{eqnarray*}
\langle \frak{T}_{\Delta,p}(t)(a), b\rangle&=&e^{-({t\over p}+1-e^{-t})}\sum_{l=0}^\infty b(l)\displaystyle\sum_{j=0}^l{l \choose j}(1-e^{-t})^{l-j}e^{-tj}\sum_{n=j}^\infty{(1-e^{-t})^{n-j}\over (n-j)!}a(n)\cr
&=&e^{-({t\over p}+1-e^{-t})}\sum_{j=0}^\infty e^{-tj} \displaystyle\sum_{l=j}^\infty{l \choose j}(1-e^{-t})^{l-j}b(l)\sum_{n=j}^\infty{(1-e^{-t})^{n-j}\over (n-j)!}a(n)\cr
&=&e^{-(t(1-{1\over p'})+1-e^{-t})}\sum_{n=0}^\infty a(n) \sum_{j=0}^n{(1-e^{-t})^{n-j}\over (n-j)!}e^{-tj} \displaystyle\sum_{l=j}^\infty{l \choose j}(1-e^{-t})^{l-j}b(l)\cr
&=&\sum_{n=0}^\infty a(n)\frak{S}_{\nabla,p}(t)b(n)=\langle a, \frak{S}_{\nabla,p}(t)(b)\rangle, \cr
\end{eqnarray*}
and we conclude the proof.
\end{proof}

\begin{remark}{\rm Note that we may express $(\frak{T}_{\Delta,p}(t))_{t>0}$ and $(\frak{S}_{\nabla,p}(t))_{t>0}$  $C_0$-semigroups
in terms of $(\frak{T}_{p}(t))_{t>0}$ and $(\frak{S}_{p}(t))_{t>0}$. In   fact
$$
\frak{T}_{\Delta,p}(t)= \frak{T}_{p}(t)\circ e^{(1-e^{-t})\Delta}, \qquad
\frak{S}_{\nabla,p}(t)=  e^{(1-e^{-t})\nabla}\circ\frak{S}_{p}(t),
$$
for $t>0$.}

\end{remark}

Given $1\le p\le \infty$ and a function $f \in L^p(\RR^+)$, we consider the function $S_{ p}(t)f$ defined by
\begin{equation}\label{semigrp}
S_{p}(t)f(s):= e^{-t\over p}f(e^{-t}s+1-e^{-t}), \qquad s>0,
\end{equation}
and $t\ge 0.$  The operators $(S_p(t))_{t>0}$ is a contractive $C_0$-semigroups of operators on $L^p(\R^+)$,  $1\le p<\infty$, which infinitesimal generator $(A_p, D(A_p))$ is defined by
$$
(A_pf)(s):=(1-s)f'(s)-{1\over p}f(s), \qquad f \in D(A_p),\quad s>0,
$$
and $D(A_p)=\{f \in L^p(\R^+) \,\, \vert \,\, A_pf\in L^p(\R^+)\},$ see \cite[Theorem 8]{MP}. For $f\in D(\Lambda_p)\cap D({d\over ds})$ then $f\in D(A_p)$ and  $A_pf=(\Lambda_p+{d\over ds})f$.

Similarly, we consider the function $R_{ p}(t)f$ defined by
\begin{equation}\label{semigrp2}
R_{p}(t)f(s):= e^{t\over p}f(e^{t}s+1-e^{t})\chi_{(1-e^{-t}, \infty)}(s), \qquad s>0,
\end{equation}
and $t\ge 0.$  The operators $(R_p(t))_{t>0}$ is also a contractive $C_0$-semigroups of operators in $L^p(\R^+)$,  $1\le p<\infty$, which infinitesimal generator $(B_p, D(B_p))$ is defined by
$$
(B_pf)(s):=(s-1)f'(s)+{1\over p}f(s), \qquad f \in D(A_p),\quad s>0,
$$
and $D(B_p)=\{f \in L^p(\R^+) \,\, \vert \,\, B_pf\in L^p(\R^+)\}.$ For $f\in D(\Lambda_p)\cap D({d^0\over ds})$ then $f\in D(B_p)$ and  $B_pf=(-\Lambda_p+{d\over ds})f$.

 We also may express $C_0$-semigroups  $(S_p(t))_{t>0}$ and $(R_p(t))_{t>0}$
in terms of the $C_0$-group $(T_p(t))_{t\in \R}$, defined in (\ref{eq:semigroups}), i.e.,
$$
S_p(t)= T_{left}(1-e^{-t})\circ T_p^+(t),  \qquad R_p(t)= T_p^-(t)\circ T_{right}(1-e^{-t}),
$$
for $t>0$.

\begin{theorem}\label{conmutante} Take $1\le p \le \infty$, the one-parameter families  $(S_p(t))_{t>0}$ and $(R_p(t))_{t>0}$   defined on $L^p(\R^+)$ in (\ref{semigrp}) and (\ref{semigrp2}) and $(\frak{T}_{\Delta,p}(t))_{t>0}$ and $(\frak{S}_{\nabla,p}(t))_{t>0}$  defined on $\ell^p$ in Definition \ref{deff}. Then
\begin{itemize}
    \item[(i)] $\frak{S}_{\nabla,p}(t)\circ \P= \P\circ R_p(t)$ for $t>0$.
     \item[(ii)] $ \P^*\circ\frak{T}_{\Delta,p}(t)=  S_p(t)\circ \P^*$ for $t>0$.
\end{itemize}
\end{theorem}
\begin{proof} (i) Take $ f\in L^p(\R^+)$. Then
\begin{eqnarray*}
&\quad&(\P(R_p(t)f))(n)=e^{t\over p}\int_{1-e^{-t}}^\infty e^{-r}{r^n\over n!}f(e^tr+1-e^t) dr\cr
 &\quad&\quad = e^{-(t(1-{1\over p})+1-e^{-t})}\int_{0}^\infty e^{-e^{-t}u}{(e^{-t}u+1-e^{-t})^n\over n!}f(u) du\cr
&\quad&\quad = e^{-(t(1-{1\over p})+1-e^{-t})}\sum_{j=0}^n{e^{-tj}\over j!}{(1-e^{-t})^{n-j}\over (n-j)!}\int_{0}^\infty e^{u(1-e^{-t})}e^{-u}u^jf(u) du\cr
&\quad&\quad = e^{-(t(1-{1\over p})+1-e^{-t})}\sum_{j=0}^n{e^{-tj}\over j!}{(1-e^{-t})^{n-j}\over (n-j)!}\sum_{l=0}^\infty {(1-e^{-t})^l\over l!}\int_{0}^\infty e^{-u}u^{j+l}f(u) du\cr
&\quad&\quad = e^{-(t(1-{1\over p})+1-e^{-t})}\sum_{j=0}^n{e^{-tj}\over j!}{(1-e^{-t})^{n-j}\over (n-j)!}\sum_{m=j}^\infty {m!(1-e^{-t})^{m-j}\over (m-j)!}\int_{0}^\infty e^{-u}{u^{m}\over m!}f(u) du\cr
&\quad&\quad = e^{-(t(1-{1\over p})+1-e^{-t})}\sum_{j=0}^n{e^{-tj}}{(1-e^{-t})^{n-j}\over (n-j)!}\sum_{m=j}^\infty {m\choose j} (1-e^{-t})^{m-j}\P f(m)\cr
&\quad&\quad =\frak{S}_{\nabla,p}(t)( \P f)(n),
\end{eqnarray*}
for $n\ge 0$.

To prove the item (ii), take $a=(a(n))_{n\ge 0}\in \ell^p$ for $1\le p \le \infty$. Then
\begin{eqnarray*}
&\quad &S_p(t)(\P^*(a))(r)= e^{-t\over p}e^{-(e^{-t} r+1-e^{-t})}\sum_{n=0}^\infty {a(n)}{(e^{-t} r+1-e^{-t})^n\over n!}\cr
&\quad&\quad = e^{-({t\over p}+1-e^{-t}+r)}\sum_{n=0}^\infty {a(n)}{(e^{-t} r+1-e^{-t})^n\over n!}e^{r(1-e^{-t})}\cr
&\quad&\quad = e^{-({t\over p}+1-e^{-t}+r)}\sum_{n=0}^\infty {a(n)}\sum_{j=0}^n{e^{-tj} \over j!}{(1-e^{-t})^{n-j}\over (n-j)!}\sum_{m=0}^\infty {r^{m+j}(1-e^{-t})^m\over m!}\cr
&\quad&\quad = e^{-({t\over p}+1-e^{-t}+r)}\sum_{j=0}^\infty {e^{-tj} \over j!}\sum_{n=j}^\infty {a(n)}{(1-e^{-t})^{n-j}\over (n-j)!}\sum_{l=j}^\infty {r^{l}(1-e^{-t})^{l-j}\over (l-j)!}\cr
&\quad&\quad = e^{-({t\over p}+1-e^{-t}+r)}\sum_{l=0}^\infty r^{l} \sum^{l}_{j=0}{e^{-tj} \over j!}{(1-e^{-t})^{l-j}\over (l-j)!}\sum_{n=j}^\infty {a(n)}{(1-e^{-t})^{n-j}\over (n-j)!}\cr
&\quad&\quad = e^{-r}\sum_{l=0}^\infty {r^{l}\over l!} e^{-({t\over p}+1-e^{-t})}\sum^{l}_{j=0}{l\choose j}e^{-tj}(1-e^{-t})^{l-j}\sum_{n=j}^\infty {a(n)}{(1-e^{-t})^{n-j}\over (n-j)!}\cr
&\quad&\quad =  \P^*(\frak{T}_{\Delta,p}(t)(a))(r),
\end{eqnarray*}
for $t>0$ and we conclude the result.
\end{proof}

We use Theorem \ref{conmutante} to conclude that $(\frak{T}_{\Delta,p}(t))_{t>0}$ and $(\frak{S}_{\nabla,p}(t))_{t>0}$  are $C_0$-semigroups in the next theorem

\begin{theorem}\label{inge} Take $1\le p < \infty$. The one-parameter families   $(\frak{T}_{\Delta,p}(t))_{t>0}$ and $(\frak{S}_{\nabla,p}(t))_{t>0}$  defined on $\ell^p$ in Definition \ref{deff} are contractive $C_0$-semigroups. Moreover

\begin{itemize}
\item[(i)] the infinitesimal generator of $(\frak{T}_{\Delta,p}(t))_{t>0}$ is the closed operator $({\frak A}_p+\Delta, D({\frak A}_p))$. We denote this operator sum by ${\frak A}_{\Delta, p}$
\item[(ii)] the infinitesimal generator of $(\frak{S}_{\nabla,p}(t))_{t>0}$ is the closed operator $({\frak B}_p+\nabla, D({\frak B}_p))$. We denote this operator sum by ${\frak B}_{\nabla, p}$
    \item[(iii)]  operators ${\frak A}_{\Delta, p}$ and  ${\frak B}_{\nabla, p'}$  and semigroups  $(\frak{T}_{\Delta,p}(t))_{t>0}$ and $(\frak{S}_{\nabla,p'}(t))_{t>0}$   are adjoint
to each other, i.e., $({\frak A}_{\Delta, p})^*={\frak B}_{\nabla, p'}$ and $(\frak{T}_{\Delta,p}(t))^*=\frak{S}_{\nabla,p'}(t)$ on $\ell^{p'}$ where $1/p + 1/p'= 1$ and $t>0$.
\end{itemize}

\end{theorem}

\begin{proof} First of all, we check the semigroup property. Take $t,s>0$ and $a=(a(n))_{n\ge 0}\in \ell^p$ for $1\le p<\infty$. By Theorem \ref{conmutante} (ii), we obtain that
\begin{eqnarray*}
 \P^*(\frak{T}_{\Delta,p}(t+s)(a))(r)&=& S_p(t+s)(\P^*(a))(r)=S_p(t)(S_p(s)(\P^*(a)))(r)\cr
 &=&S_p(t)(\P^*(\frak{T}_{\Delta,p}(s)(a))(r)= \P^*( \frak{T}_{\Delta,p}(t)\frak{T}_{\Delta,p}(s)(a))(r),
\end{eqnarray*}
for $r>0$. By Proposition \ref{conne} (ii), we obtain that $\frak{T}_{\Delta,p}(t+s)(a)=  \frak{T}_{\Delta,p}(t)\frak{T}_{\Delta,p}(s)(a)$.  To show the strongly continuity, it is enough to check that $\frak{T}_{\Delta,p}(t)(a)\to a$ when $t\to 0$ on $\ell^1$ and conclude the result  on $\ell^p$ by density. Then
\begin{eqnarray*}
&\quad&\Vert \frak{T}_{\Delta,p}(t)(a)- a\Vert_1\cr
&\quad&\qquad= \sum_{l=0}^\infty \lvert e^{-({t\over p}+1-e^{-t})}\displaystyle\sum_{j=0}^l{l \choose j}(1-e^{-t})^{l-j}e^{-tj}\sum_{n=j}^\infty{(1-e^{-t})^{n-j}\over (n-j)!}a(n)- a(l)\rvert\cr
&\quad&\qquad\le\sum_{l=0}^\infty  \displaystyle\sum_{j=0}^l{l \choose j}(1-e^{-t})^{l-j}e^{-tj}\lvert e^{-({t\over p}+1-e^{-t})} \sum_{n=j}^\infty{(1-e^{-t})^{n-j}\over (n-j)!}a(n)- a(l)\rvert.
\end{eqnarray*}
Note that $(1-e^{-t})^{l-j}\to 0$ and $(1-e^{-t})^{n-j}\to 0$ for $l,n>j$  and $e^{-({t\over p}+1-e^{-t})}a(l)-a(l)\to 0$, in the case that $l=n=j$ when $t\to 0$. By Theorem \ref{norms} (i), we conclude that $(\frak{T}_{\Delta,p}(t))_{t>0}$  is a contractive $C_0$-semigroups of operators on $\ell^p$ for $1\le p<\infty$. Similarly, we prove that $(\frak{S}_{\nabla,p}(t))_{t>0}$ is also a contractive $C_0$-semigroups of operators on $\ell^p$ for $1\le p<\infty$, where we use Proposition \ref{conne} (iii).

\noindent (i) Now, we check the infinitesimal generator of  $(\frak{T}_{\Delta,p}(t))_{t>0}$. Take  $a=(a(n))_{n\ge 0}\in \ell^p$ and $l\ge 1$,
\begin{eqnarray*}
&\quad&{\partial \over \partial t}\left(\displaystyle\sum_{j=0}^l{l \choose j}(1-e^{-t})^{l-j}e^{-tj}\sum_{n=j}^\infty{(1-e^{-t})^{n-j}\over (n-j)!}a(n)\right)\cr
&\quad&=\sum_{j=0}^{l-1}{l \choose j}(l-j)(1-e^{-t})^{l-j}e^{-t(j+1)}\sum_{n=j}^\infty{(1-e^{-t})^{n-j}\over (n-j)!}a(n)\cr
&\quad&\qquad +\sum_{j=1}^l{l \choose j}(1-e^{-t})^{l-j}(-j)e^{-tj}\sum_{n=j}^\infty{(1-e^{-t})^{n-j}\over (n-j)!}a(n)\cr
&\quad&\qquad +\sum_{j=0}^l{l \choose j}(1-e^{-t})^{l-j}e^{-tj}\sum_{n=j+1}^\infty{(1-e^{-t})^{n-j-1}\over (n-j-1)!}a(n)
\end{eqnarray*}
for $t>0$. Then
\begin{eqnarray*}
\lim_{t\to 0}{\frak{T}_{\Delta,p}(t)(a)(l)-a(l)\over t}&=& \left.{\partial \over \partial t}\frak{T}_{\Delta,p}(t)(a)(l)\right\vert_{t=0}= -(1+{1\over p})a(l) +l a(l-1)-la(l)+a(l+1)\cr
&=&l(a(l-1)-a(l))-{1\over p}a(l)+a(l+1)-a(l)= ({\frak A}_p+\Delta)(a)(l),
\end{eqnarray*}
for $a\in D({\frak A}_p)$. Note that $({\frak A}_p+\Delta)(a)(0)=a(1)-(1+{1\over p})a(0)$ for $a\in D({\frak A}_p)$.

\noindent (ii) Now we check the infinitesimal generator of  $(\frak{S}_{\Delta,p}(t))_{t>0}$. Take  $b=(b(n))_{n\ge 0}\in \ell^p$ and $l\ge 1$,
\begin{eqnarray*}
&\quad&{\partial \over \partial t}\left(\displaystyle\sum_{j=0}^l{(1-e^{-t})^{l-j}\over (l-j)!}e^{-tj}\sum_{n=j}^\infty{n \choose j}(1-e^{-t})^{n-j}b(n)\right)\cr
&\quad&=\displaystyle\sum_{j=1}^l{(1-e^{-t})^{l-j}\over (l-j)!}(-j)e^{-tj}\sum_{n=j}^\infty{n \choose j}(1-e^{-t})^{n-j}b(n)\cr
&\quad&\qquad +\displaystyle\sum_{j=0}^{l-1}{(1-e^{-t})^{l-1-j}\over (l-1-j)!}e^{-tj}\sum_{n=j}^\infty{n \choose j}(1-e^{-t})^{n-j}b(n)\cr
&\quad&\qquad +\displaystyle\sum_{j=0}^l{(1-e^{-t})^{l-j}\over (l-j)!}e^{-tj}\sum_{n={j+1}}^\infty{n \choose j}(1-e^{-t})^{n-j}(l-j)b(n)
\end{eqnarray*}
for $t>0$. Then we get that
\begin{eqnarray*}
&\quad&\lim_{t\to 0}{\frak{S}_{\Delta,p}(t)(b)(l)-b(l)\over t}= \left.{\partial \over \partial t}\frak{S}_{\Delta,p}(t)(b)(l)\right\vert_{t=0}\cr
&\quad&\quad= -(2-{1\over p})b(l)-lb(l) +b(l-1)+(l+1)b(l+1)\cr
&\quad&\quad =(l+1)b(l+1)-lb(l)-(1-{1\over p})b(l)+b(l-1)-b(l)= ({\frak B}_p+\nabla)(b)(l),
\end{eqnarray*}
for $b\in D({\frak B}_p)$. Note that $({\frak B}_p+\nabla)(b)(0)=b(1)-(2-{1\over p})b(0)$ for $b\in D({\frak B}_p)$.

 The item (iii) is proven in Theorem \ref{norms} (iii). \end{proof}

 Since $C_0$-semigroups $(\frak{T}_{\Delta,p}(t))_{t>0}$ and $(\frak{S}_{\nabla,p}(t))_{t>0}$  are generated by bounded perturbations of infinitesimal generators of $(\frak{T}_{p}(t))_{t>0}$ and $(\frak{S}_{p}(t))_{t>0}$, we apply the variation of
parameters formula for the perturbed semigroup to get the following corollary, see \cite[Corollary III.1.7]{En-Na-00}.

 \begin{corollary} Take $1\le p < \infty$, $a\in \ell^p$  and  $(\frak{T}_{\Delta,p}(t))_{t>0}$ and $(\frak{S}_{\nabla,p}(t))_{t>0}$ the contractive $C_0$-semigroups defined on $\ell^p$ in Definition \ref{deff}.
 \begin{itemize}
 \item[(i)] For $t>0$ and $n\ge 0$, we have that
 $$
 \frak{T}_{\Delta,p}(t)(a)(n)= \frak{T}_{p}(t)(a)(n)+\int_0^t \frak{T}_{p}(t-s)\left( \frak{T}_{\Delta,p}(s)(a)(n+1)-\frak{T}_{\Delta,p}(s)(a)(n) \right)ds.
 $$
 \item[(ii)] For $t>0$ and $n\ge 1$, we have that
 $$
 \frak{S}_{\nabla,p}(t)(a)(n)= \frak{S}_{p}(t)(a)(n)+\int_0^t \frak{S}_{p}(t-s)\left( \frak{S}_{\nabla,p}(s)(a)(n-1)-\frak{S}_{\Delta,p}(s)(a)(n) \right)ds,
 $$
 and
 $$
 \frak{S}_{\nabla,p}(t)(a)(0)= \frak{S}_{p}(t)(a)(0)-\int_0^t \frak{S}_{p}(t-s)\frak{S}_{\Delta,p}(s)(a)(0) ds.
 $$
 \end{itemize}
 \end{corollary}

 In the next theorem, we present the spectral study of operators  ${\frak A}_{\Delta, p}$ and ${\frak B}_{\nabla, p}$.

 \begin{theorem}\label{resolvent} Take $1\le p < \infty$,  ${\frak A}_{\Delta, p}$ and ${\frak B}_{\nabla, p}$ the infinitesimal generators of  $(\frak{T}_{\Delta,p}(t))_{t>0}$ and $(\frak{S}_{\nabla,p}(t))_{t>0}$  on $\ell^p$, $1\le p<\infty$ respectively.
 \begin{itemize}
\item[(i)] Then $\sigma_{point}( {\frak A}_{\Delta, p})= \emptyset$ and $\sigma_{point}( {\frak B}_{\nabla, p})= \CC_{-}$.
\item[(ii)] The following equalities hold  $\sigma ( {\frak A}_{\Delta, p})= \sigma ({\frak B}_{\nabla, p})= \overline{\CC_-}.$
\item[(iii)] For $\Re \lambda >0$, and $a=(a(n))_{n\ge 0} \in \ell^p$, we have that
\begin{eqnarray*}
(\lambda- {\frak A}_{\Delta, p})^{-1}a(l)&=& \sum_{j=0}^l{l\choose j}\sum_{n=j}^\infty {a(n)\over (n-j)!}\B_1(\lambda+{1\over p}+j, n+l-2j+1),\cr
(\lambda- {\frak B}_{\nabla, p})^{-1}a(l)&=& \sum_{j=0}^l{1\over(l-j)! }\sum_{n=j}^\infty {n\choose j}a(n)\B_1(\lambda+1-{1\over p}+j, n+l-2j+1),\cr
\end{eqnarray*}
where the function $\B_1$ is given in Definition \ref{Chat}.
 \end{itemize}
 \end{theorem}

 \begin{proof} (i) Take $\lambda \in \CC$ and $a=(a(n))_{n\ge 0} \in \ell^p$ such that ${\frak A}_{\Delta, p}a(n)=\lambda a(n)$ for $n\ge 0$, i.e.,  ${\frak A}_{\Delta, p}a(0)=a(1)-(1+{1\over p}) a(0)= \lambda a(0)$ and
 $$
 a(n+1) -(n+1+{1\over p})a(n) +na(n-1)=\lambda a(n), \quad  n\ge 1,
 $$
 Then we have that $a(1)= (\lambda + 1+{1\over p})a(0)$ and
 $$
 a(n+1)= (\lambda+ {1\over p}+n+1)a(n)-na(n-1), \quad  n\ge 1.
 $$
 By Theorem \ref{Charlier} (ii), we conclude that $a(n)= p_n(\lambda+{1\over p}) a(0)$. In the case that $a(0)\not=0$, we conclude that $(a(n))_{n\ge 0}\not \in \ell^p$.

 Now take $1\le p\le \infty$, $\lambda \in \CC$ and $b=(b(n))_{n\ge 0} \in \ell^p$ such that ${\frak B}_{\nabla, p}b(n)=\lambda b(n)$ for $n\ge 0$, i.e., ${\frak B}_{\nabla, p}b(0)=b(1)-(2-{1\over p})b(0)= \lambda b(0)$ and
 $$
 (n+1)b(n+1)-(n+2-{1\over p})b(n)+b(n-1)=\lambda b(n), \quad  n\ge 1.
$$
Then, we deduce $b(1)=(2-{1\over p}+\lambda)b(0)$ and
$$
 (n+1)b(n+1)=(n+1+\lambda+1-{1\over p})b(n)-b(n-1), \quad  n\ge 1.
$$
By Theorem \ref{qn}(i) and (ii), $b(n)= q_n(\lambda+1-{1\over p})b(0)$ and $b\in \ell^p$ if and only if $b(0)=0$ or $\Re\lambda+1-{1\over p}<1-{1\over p}$, i.e. $\Re\lambda<0$ for $1\le p\le \infty$ and $\lambda \in \CC^-$.

\noindent (ii) Since the operator $({\frak B}_{\nabla, p}, D({\frak B}_{\Delta, p}))$ generates the contraction $C_0$-semigroup $(\frak{S}_{\nabla,p}(t))_{t>0}$ (Theorem \ref{norms} (ii) and \ref{inge} (ii)) then we conclude that $\sigma ({\frak B}_{\Delta, p})\subset \overline{\CC_-}$ for $1\le p<\infty$. As $\sigma ({\frak B}_{\nabla, p})$ is closed, the item (i) implies $\sigma ({\frak B}_{\nabla, p}) =\overline{\CC_-}$ for $1\le p<\infty$.  Note that $({\frak A}_{\Delta, p})^*={\frak B}_{\nabla, p'}$  for ${1\over p}+{1\over p'}=1$. By \cite[Proposition IV.2.18 (i)]{En-Na-00}, we conclude
$ \sigma ( {\frak A}_{\Delta, p})= \sigma ({\frak B}_{\nabla, p'})= \overline{\CC_-} $ for $1<p<\infty$. For $p=1$, note  $({\frak A}_{\Delta, 1})^*={\frak B}_{\nabla, \infty}$ and $\sigma({\frak A}_{\Delta, 1})=\sigma( {\frak B}_{\nabla, \infty})$, see \cite[Corollary B.12]{En-Na-00}. By the item (i), $\CC_-\subset \sigma( {\frak B}_{\nabla, \infty})=\sigma({\frak A}_{\Delta, 1})\subset \overline{\CC_-}$. Since the set $\sigma({\frak A}_{\Delta, 1})$ is closed, we conclude that $\sigma({\frak A}_{\Delta, 1})=\overline{\CC_-}$.

\noindent (iii) Now $\lambda \in \CC$ with $\Re \lambda>0$ and $a=(a(n))_{n\ge 0} \in \ell^p$. Then
\begin{eqnarray*}
(\lambda- {\frak A}_{\Delta, p})^{-1}a(l)&=& \int_0^\infty e^{-\lambda t}\frak{T}_{\Delta,p}(t)a(l)dt\cr
&=&\int_0^\infty e^{-(t(\lambda + {1\over p})+1-e^{-t})}\displaystyle\sum_{j=0}^l{l \choose j}(1-e^{-t})^{l-j}e^{-tj}\sum_{n=j}^\infty{(1-e^{-t})^{n-j}\over (n-j)!}a(n)\cr
&=&\displaystyle\sum_{j=0}^l{l \choose j}\sum_{n=j}^\infty{a(n)\over (n-j)!}
\int_0^\infty (1-e^{-t})^{l+n-2j} e^{-t(\lambda + {1\over p}+j)}e^{-(1-e^{-t})}dt\cr
&=& \sum_{j=0}^l{l\choose j}\sum_{n=j}^\infty {a(n)\over (n-j)!}\B_1(\lambda+{1\over p}+j, n+l-2j+1),\cr
\end{eqnarray*}
where we make the change of variable $r=1-e^{-t}$ to obtain the last equality. Similarly we show that
$$
(\lambda- {\frak B}_{\nabla, p})^{-1}a(l)=\sum_{j=0}^l{1\over(l-j)! }\sum_{n=j}^\infty {n\choose j}a(n)\B_1(\lambda+1-{1\over p}+j, n+l-2j+1),
$$
and we conclude the proof. \end{proof}

\begin{remark}{\rm As operators ${\frak A}_{\Delta, p}$ and ${\frak B}_{\nabla, p}$  are bounded perturbations of ${\frak A}_{ p}$ and ${\frak B}_{p}$, we have that
\begin{eqnarray*}
(\lambda- {\frak A}_{\Delta, p})^{-1}&=& (\lambda- {\frak A}_{ p})^{-1}\sum_{n=0}^\infty (\Delta (\lambda- {\frak A}_{ p})^{-1})^n, \cr
(\lambda- {\frak B}_{\nabla, p})^{-1}&=& (\lambda- {\frak B}_{ p})^{-1}\sum_{n=0}^\infty (\nabla (\lambda- {\frak B}_{ p})^{-1})^n,
\end{eqnarray*}
for $\Re \lambda >2$ and $1\le p < \infty$, see for example \cite[(III.1.3)]{En-Na-00}.}
\end{remark}

\section{Ces\`aro-like operators  on $\ell^p$}

Let $(X, \Vert \quad \Vert)$ be a Banach space and $T=(T(t))_{t>0}$ a uniformly bounded $C_0$-semigroup on $(X, \Vert \quad \Vert)$. Then we may define the bounded operator
\begin{equation}\label{cesaro}
C^T_{\mu, \nu}(x):=\int_0^\infty e^{-\mu t} (1-e^{-t})^{\nu-1}T({t})(x)dt, \qquad x\in X,
\end{equation}
for $\mu, \nu>0$. The author and other collaborators have followed this point of view in \cite[Section 7]{AM2018}, \cite[Section 3]{LMPS} and \cite[Section 4]{MP} for concrete $C_0$-semigroups in particular Lebesgue function spaces and sequence spaces $(X, \Vert \quad \Vert)$. In this section, we consider the operator $C_{\mu, \nu}$ for some  Koopman semigroups $(\frak{T}_{p}(t))_{t>0}$ and $(\frak{S}_{p}(t))_{t>0}$  and perturbed Koopman semigroups  $(\frak{T}_{\Delta,p}(t))_{t>0}$ and $(\frak{S}_{\nabla,p}(t))_{t>0}$ studied in the section above.  Via the Poisson transformations $\mathcal P$ and $\mathcal P^*$, we give connection with some integral operators. We also estimate norms of these Ces\`aro-like operators and identify their spectrum sets.

\subsection{Ces\`aro-like operators subordinated to  Koopman semigroups on  $\ell^p$}

Let $\Re \alpha > 0$, consider the Cesàro-Hardy operator of order $\alpha$ given by
\[
    \Ca f (s) := \frac{\alpha}{s^\alpha} \int_0^s (s-u)^{\alpha-1} f(u) \, du, \quad s\ge 0,
\]
and the dual Cesàro-Hardy operator of order $\alpha$ given by
\[
    \Ca^* f (s): = \alpha \int_s^\infty\frac{(u-s)^{\alpha-1}}{u^\alpha} f(u) \, du, \quad s\ge 0.
\]
 For $ \Ca$ the change of variables $u = e^{-t}s$ yields
\[
    \Ca f (s) = \alpha \int_0^\infty (1-e^{-t})^{\alpha-1}e^{-t(1-{1\over p})} T^+_p(t)f(s)\, dt, \qquad s>0,
\]
and for $\Ca^*$ the change of variable $u = e^{t}s$ gives
\[
    \Ca^* f (s) =  \alpha \int_0^\infty (1-e^{-t})^{\alpha-1}e^{-{t\over p}} T^-_p(t)f(s) \, dt, \qquad s>0,
\]
where $C_0$-semigroups $(T^+_p(t))_{t \geq 0}$ and $T^-_p(t))_{t \geq 0}$ are defined above, see \cite[Theorem 3.3 and 3.7]{LMPS}.

 Now we consider the generalized discrete  Ces\`{a}ro operator of order $\alpha>0$ given for
\begin{eqnarray*}
\frak{C}_{\alpha}a(n)&:=&\frac{1}{k^{\alpha+1}(n)}\sum_{j=0}^nk^{\alpha}(n-j)a(j),  \cr \frak{C}^*_{\alpha}a(n)&:=&\sum_{j=n}^{\infty}\frac{1}{k^{\alpha+1}(j)}k^{\alpha}(j-n)a(j),
\end{eqnarray*}
for $n \in \N_0$, $a\in \ell^p$ and $(k^\alpha(n))_{n\ge 0}$ are the  Cesàro numbers. Remind that for $a\in \ell^p$ and $n\in\N_0$, we have that
\begin{eqnarray*}
{\frak{C}}_{\alpha}a(n)&=&\displaystyle\alpha\int_0^{\infty}(1-e^{-t})^{\alpha-1}e^{-t(1-{1\over p})} \frak{T}_p(t)a(n)\,dt,\quad \, 1<p\leq\infty,\cr
{\frak{C}}^*_{\alpha}a(n)&=&\alpha\int_0^{\infty}(1-e^{-t})^{\alpha-1}e^{-{t\over p}}\frak{S}_p(t)a(n)\,dt,\quad \, 1\leq p <\infty,
\end{eqnarray*}
see \cite[Theorem 7.2] {AM2018}.

The following corollary presents the connection between  ${\mathcal{C}}_{\alpha}$, ${\mathcal{C}}_{\alpha}^*$,  $\frak{C}_{\alpha}$  and $\frak{C}_{\alpha}^*$. The proof is a straightforward from Theorem \ref{eliz}.

\begin{corollary}\label{kkl} Take  $\alpha >0$. Then
\begin{itemize}
\item[(i)] ${\mathcal P}\circ  {\mathcal{C}}_{\alpha}^*=\frak{C}^*_{\alpha}\circ {\mathcal P}$ for  $1\leq p <\infty$.
\item[(ii)] ${\mathcal P}^*\circ\frak{C}_{\alpha}=  {\mathcal{C}}_{\alpha}\circ {\mathcal P}^*$ for  $1<p\leq\infty$.
\end{itemize}
\end{corollary}

\begin{remark} {\rm Take $\alpha>0$, Then $\partial \sigma(\frak{C_{\alpha}})= \sigma(\mathcal{C_{\alpha}})$ due to
\begin{eqnarray*}
\sigma(\mathcal{C_{\alpha}})&=&\overline{\biggl\{ \frac{\Gamma(\alpha+1)\Gamma(it+1-\frac{1}{p})}{\Gamma(\alpha+it+1-\frac{1}{p})}\,:\, t\in\R \biggr\}}\cr
\sigma(\frak{C_{\alpha}})&=& \overline{\biggl\{ \frac{\Gamma(\alpha+1)\Gamma(z+1-\frac{1}{p})}{\Gamma(\alpha+z+1-\frac{1}{p})}\,:\, z\in\CC_+\cup i\R \biggr\}},
\end{eqnarray*} for $1<p\leq\infty, $
 see \cite[Theorem 3.5]{LMPS} and \cite[Theorem 7.6 (i)]{AM2018}. Similarly $\partial \sigma(\frak{C_{\alpha}^*})= \sigma(\mathcal{C_{\alpha}^*})$, due to
 \begin{eqnarray*}
 \sigma(\mathcal{C^*_{\alpha}})&=&\overline{\biggl\{ \frac{\Gamma(\alpha+1)\Gamma(it+\frac{1}{p})}{\Gamma(\alpha+it+\frac{1}{p})}\,:\, t\in\R \biggr\}},\cr
 \sigma(\frak{C_{\alpha}^*})&=&\overline{\biggl\{ \frac{\Gamma(\alpha+1)\Gamma(z+\frac{1}{p})}{\Gamma(\alpha+z+\frac{1}{p})}\,:\, z\in\CC_+\cup i\R \biggr\}},\cr
\end{eqnarray*}  for $1\leq p <\infty,$   see Theorem \cite[Theorem 3.9]{LMPS}  and \cite[Theorem 7.6(ii)]{AM2018}.
}
\end{remark}

\subsection{Ces\`aro-like operators subordinated to perturbed Koopman semigroups on  $\ell^p$}

 In this section, we identify the operator $C_{\mu, \nu}$ for perturbed  Koopman semigroups  $(\frak{T}_{\Delta,p}(t))_{t>0}$ and $(\frak{S}_{\nabla,p}(t))_{t>0}$.

\begin{theorem} Take $1\le p < \infty$, $a=(a(n))_{n\ge 0}\in \ell^p$ and $C_0$-semigroups  $(\frak{T}_{\Delta,p}(t))_{t>0}$ and $(\frak{S}_{\nabla,p}(t))_{t>0}$  on $\ell^p$ given in Definition \ref{deff}. Now we define

\begin{eqnarray*}
\frak{c}^{\Delta,p}_{\mu, \nu}a(l)&:=&\int_0^\infty e^{-\mu t} (1-e^{-t})^{\nu-1}\frak{T}_{\Delta,p}(t)a(l)dt,\cr
\frak{c}^{\nabla,p}_{\mu, \nu}a(l)&:=&\int_0^\infty e^{-\mu t} (1-e^{-t})^{\nu-1}\frak{S}_{\Delta,p}(t)a(l)dt,
\end{eqnarray*}
for $\Re\mu,\Re \nu>0$ and $l\ge 0$.
\begin{itemize}
\item[(i)] Let the function $\B_1$ give in Definition \ref{Chat} and $l\ge 0$. Then
\begin{eqnarray*}
\frak{c}^{\Delta,p}_{\mu, \nu}a(l)&=& \sum_{j=0}^l{l\choose j}\sum_{n=j}^\infty {a(n)\over (n-j)!}\B_1(\mu+{1\over p}+j, \nu+ n+l-2j),\cr
\frak{c}^{\nabla,p}_{\mu, \nu}a(l)&=& \sum_{j=0}^l{1\over(l-j)! }\sum_{n=j}^\infty {n\choose j}a(n)\B_1(\mu+1-{1\over p}+j, \nu+n+l-2j).
\end{eqnarray*}
\item[(ii)] Operators $\frak{c}^{\Delta,p}_{\mu, \nu}$ and $\frak{c}^{\nabla,p}_{\mu, \nu}$ are bounded on $\ell^p$,
$$
 \Vert \frak{c}^{\Delta,p}_{\mu, \nu}\Vert, \, \Vert \frak{c}^{\nabla,p}_{\mu, \nu}\Vert\le \B(\Re\mu, \Re\nu),
$$
and $(\frak{c}^{\Delta,p}_{\mu, \nu})^*=\frak{c}^{\nabla,p'}_{\mu, \nu}$ for   $1<p<\infty$ and ${1\over p}+{1\over p'}=1$.

\item[(iii)] The spectrum of $\frak{c}^{\Delta,p}_{\mu, \nu}$ and $\frak{c}^{\nabla,p}_{\mu, \nu}$ on ${\mathcal B}(\ell^p)$ is the set
$$
\sigma(\frak{c}^{\Delta,p}_{\mu, \nu})=\sigma(\frak{c}^{\nabla,p}_{\mu, \nu})= \overline{\biggl\{ \B(\mu+z,\nu)\,:\, \Re z\ge 0 \biggr\}}.
$$
\end{itemize}
 \end{theorem}

\begin{proof} Note that the item (i) is an  extension of Theorem \ref{resolvent} (iii) for $\nu=1$  and the proof follows similar ideas. The proof of (ii) is  straightforward from definitions of $\frak{c}^{\Delta,p}_{\mu, \nu}$ and $\frak{c}^{\nabla,p}_{\mu, \nu}$, Theorem \ref{norms} (i) and (ii) and Theorem \ref{inge}(iii). Finally to show item (iii), we apply  similar ideas than in the proof of \cite[Theorem 16 (iii)]{MP}.
\end{proof}

On $L^p(\RR^+)$ with $1\le p<\infty$, we consider contractive $C_0$-semigroups  $S_p=(S_p(t))_{t>0}$ and $R_p=(R_p(t))_{t>0}$ defined by (\ref{semigrp}) and (\ref{semigrp2}). We also define {Ces\`aro-like operators subordinated to these semigroups on  $L^p(\RR^+)$ using (\ref{cesaro}) to get
\begin{eqnarray*}
C_{\mu, \nu}^{S_p}f(r)&=&{1\over\vert r-1\vert^{\mu+{1\over p}+\nu-1}}
\int_{\Gamma_{1,r}}\vert s-1\vert^{\mu+{1\over p}-1} \vert r-s\vert^{\nu -1} f(s)ds, \cr
C_{\mu, \nu}^{R_p}f(r)&=&\vert r-1\vert^{\mu-{1\over p}}
\int_{\Gamma'_{1,r}} {\vert r-s\vert^{\nu -1}1\over\vert s-1\vert^{\mu-{1\over p}+\nu}}f(s)ds
\end{eqnarray*}
for $r>0$ and   $\Gamma_{1,r}:=(1,r)$ for $r>1$ and $\Gamma_{1,r}:=(r,1)$ in the case $0<r<1$ and $\Gamma'_{1,r}:=(r,+\infty)$ for $r>1$ and $\Gamma'_{1,r}:=(r,1)$ in the case $0<r<1$ ({\cite[Theorem 16(i)]{MP}). These integral operators are known as Chen fractional integral, see \cite[Section 18.5]{Sa-Ki-Ma-93}.

The following corollary presents the connection between  $\frak{c}^{\Delta,p}_{\mu, \nu}$,  $\frak{c}^{\nabla,p}_{\mu, \nu}$,  $C_{\mu, \nu}^{S_p}$  and $C_{\mu, \nu}^{R_p}$ via the Poisson transformations $\mathcal P$ and $\mathcal P^*$. The proof is straightforward from Theorem \ref{conmutante}.

\begin{corollary}\label{zzz} Take  $\alpha >0$, $\Re \mu,$ and $\Re \nu >0$. Then
\begin{itemize}
\item[(i)] ${\mathcal P}\circ  C_{\mu, \nu}^{R_p}=\frak{c}^{\nabla,p}_{\mu, \nu}\circ {\mathcal P}$ for  $1\leq p <\infty$.
\item[(ii)] ${\mathcal P}^*\circ\frak{c}^{\Delta,p}_{\mu, \nu}=  C_{\mu, \nu}^{S_p}\circ {\mathcal P}^*$ for  $1\leq p <\infty$.
\end{itemize}
\end{corollary}

\subsection*{Acknowledgements} The author thanks Luciano Abadias, Jos\'e A. Adell, Alejandro Mahillo and Jes\'us Oliva-Maza for  several comments, corrections and suggestions that lead  a improved version of this paper.


\end{document}